\documentclass[a4paper]{article}
\usepackage[T1]{fontenc}
\usepackage[utf8]{inputenc}

\usepackage{amsmath,amsthm,amssymb}
\usepackage{geometry}
\usepackage{color}
\usepackage{cite}

\usepackage{tikz}
\usepackage{pgfplots}
\usepackage{subdepth}
\usepackage{enumerate}
\usepackage{authblk}
\usepackage{mathdots}
\usepackage{mathtools}

\usepackage{algorithm}
\usepackage{algpseudocode}

\usepackage{verbatim}
\usepackage{hyperref}
\usepackage{amsmath,amsthm,amssymb}
\usepackage{color}
\newtheorem{theorem}{Theorem}[section]
\newtheorem{lemma}[theorem]{Lemma}
\newtheorem{remark}[theorem]{Remark}

\newtheorem{definition}[theorem]{Definition}

\newcommand{\R}{\mathbb R}
\newcommand{\C}{\mathbb C}
\newcommand{\ii}{\mathfrak i}

\newcommand{\la}{\lambda}

\newcommand{\diag}{{\rm diag}}

\newcommand{\rank}{{\rm rank\,}}

\usepackage{tcolorbox}
\tcbuselibrary{skins}
\tcbuselibrary{breakable}
\tcbset{shield externalize}

\title{Real-congruence canonical forms of real matrices}
\author{Fernando De Ter\'an\thanks{Universidad Carlos III de Madrid, ROR: https://ror.org/03ths8210, Departamento de Matemáticas, Avda. de la Universidad 30, 28911 Legan\'es (Madrid), Spain, {\tt fteran@math.uc3m.es}} \quad and \quad Froil\'an M. Dopico\thanks{Universidad Carlos III de Madrid, ROR: https://ror.org/03ths8210, Departamento de Matemáticas, Avda. de la Universidad 30, 28911 Legan\'es (Madrid), Spain, {\tt dopico@math.uc3m.es}}}
\date{}

\begin{document}

\maketitle

\begin{abstract}
     We present two new canonical forms for real congruence of a real square matrix $A$. The first one is a direct sum of canonical matrices of four different types and is obtained from the canonical form under $^*$congruence of complex matrices provided by Horn and Sergeichuk in [Linear Algebra Appl. 416 (2006) 1010-1032].  The second one is a direct sum of canonical matrices of three different types, has a block tridiagonal structure and is obtained from the canonical form under $^*$congruence of complex matrices provided by Futorny, Horn and Sergeichuk in [J. Algebra 319 (2008) 2351-2371]. A detailed comparison between both canonical forms is also presented, as well as their relation with the real Kronecker canonical form under strict real equivalence of the matrix pair $(A^\top , A)$. Another canonical form for real congruence was presented by Lee and Weinberg in [Linear Algebra Appl. 249 (1996) 207-215], which consists of a direct sum of eight different types of matrices. In the last part of the paper, we explain the correspondence between the blocks in this canonical form and those in the two new forms introduced in this work.
\end{abstract}


\bigskip


\section{Introduction}

The classification of general sesquilinear or bilinear forms ${\cal A}:{\mathbb F}^n\times{\mathbb F}^n\rightarrow{\mathbb F}$ over a field $\mathbb F$ is a long-standing issue that has been addressed both from the scope of pure algebra and linear algebra. A classification of bilinear forms over the complex field (namely, when $\mathbb F=\C$) is known since, at least, the 1930s \cite[p. 139]{ta} using a linear algebra approach (see \cite{dt16} for some historical details). With a  pure algebra approach, relevant contributions were obtained in \cite{gabriel,riehm}, for sesquilinear and bilinear forms over arbitrary fields $\mathbb F$ by reducing the problem to classifying Hermitian forms over finite extensions of $\mathbb F$ (see also the Introduction of \cite{fhs08} for more details). When $\mathbb F$ is, respectively, the complex field or the real field, it is natural to look for a classification using, respectively, the $^*$congruence or the real-congruence of matrices, namely the following actions of the groups ${\rm GL}_n(\C)$ and ${\rm GL}_n(\R)$ on the sets of complex and real $n\times n$ matrices, respectively:
$$
\begin{array}{ccc}
     \begin{array}{ccc}
         {\rm GL}_n(\C)\times\C^{n\times n}&\rightarrow&\C^{n\times n} \\
          (P,A)&\mapsto&PAP^*
     \end{array}&  \mbox{and}&
     \begin{array}{ccc}
         {\rm GL}_n(\R)\times\R^{n\times n} &\rightarrow&\R^{n\times n}  \\
          (P,A)&\mapsto&PAP^\top,
     \end{array}
\end{array}
$$
where $P^*$ and $P^\top$ denote the conjugate-transpose and the transpose of $P$, respectively.
The reason for using such actions relies on the fact that two $^*$\!congruent (respectively, real-congruent) complex (resp., real) matrices correspond to the same sesquilinear (resp., bilinear) form over $\C^n\times\C^n$ (resp., $\R^n\times\R^n$) in different bases. Moreover, it is also natural to classify these sesquilinear or bilinear forms by means of a canonical form, which provides a unique representative for every equivalence class (namely, for every single {\em orbit} under the action of the group ${\rm GL}_n (\C)$ or ${\rm GL}_n (\R)$).

Canonical forms for $^*$congruence of complex matrices have been provided in \cite[Theorem 1.1(b)]{hs06} and later in \cite[Theorem 1.2]{fhs08}. In both cases, the canonical forms consist of a direct sum of canonical blocks. In the case of the canonical form in \cite{hs06} there are three different types of blocks, whereas in the one in \cite{fhs08} there are two different types of blocks, which are tridiagonal. Some of the blocks depend on certain parameters. From these canonical forms, and specializing to real matrices, it is possible to get canonical forms for real congruence, and this is the main goal of this work. In particular, we derive and present in Theorem \ref{main_th} a canonical form for real congruence of real square matrices from the canonical form introduced in \cite{hs06}. This canonical form  for real congruence consists of a direct sum of blocks of four different kinds, which come from considering separately those blocks of the canonical form for $^*$congruence of general complex matrices involving real parameters and those associated to (pairs of conjugate) non-real parameters. In terms of bilinear forms, the canonical form of Theorem \ref{main_th} allows us to classify bilinear forms over the real field. Analogously, we present and derive in Theorem \ref{second_main_th} a canonical form for real congruence of real square matrices from the canonical form introduced in \cite{fhs08}. The canonical form in Theorem \ref{second_main_th} is a direct sum of blocks of three different types. The relation between the two canonical forms for real congruence in Theorems \ref{main_th} and \ref{second_main_th} is presented in Theorem \ref{thm.relation2}, which is connected with Theorem \ref{thm.relation1} about the correspondence between the blocks of the canonical forms for $^*$\!congruence of complex matrices in \cite[Theorem 1.1(b)]{hs06} and \cite[Theorem 1.2]{fhs08}.

Another canonical form for real congruence of real square matrices was presented in \cite[Theorem II]{lw96}. It consists also of a direct sum of canonical blocks of several types, but in this case the number of different types of blocks is eight, which is twice the number of different types of blocks in the canonical form in Theorem \ref{main_th} and more than twice the number of canonical blocks in the canonical form in Theorem \ref{second_main_th}. In Section \ref{leewein_sec}, we obtain and display the correspondence between the eight types of blocks in the canonical form of \cite{lw96} and the ones presented in Theorem \ref{main_th} (the correspondence with the blocks in Theorem \ref{second_main_th} can be obtained via Theorem \ref{thm.relation2}). Moreover, in Section \ref{leewein_sec}, we impose restrictions on the values of some of the parameters appearing in the blocks in \cite[Theorem II]{lw96} to make these blocks ``truly'' canonical (see Remark \ref{rem.leeweinparam}).

The real congruence canonical form of a matrix $A\in\R^{n\times n}$ is naturally related with the real Kronecker canonical form under strict real equivalence \cite{lr-siamrev-05} of the real matrix pair $(A^\top , A)$, or, equivalently, of the matrix pencil $\la A^\top + A$ in the variable $\la$. These particular pencils are called real palindromic pencils and are in the intersection of the families of complex $\top$-palindromic pencils and $*$-palindromic pencils \cite{4m-good}. Although palindromic pencils (without the name) exist in the literature since long time ago, they have received considerable attention in the last two decades, in particular since \cite{4m-good} was published. A small sample of other recent works dealing with palindromic pencils are \cite{dt16,dt18,4m-palinsmith}. Recall in this context that two complex matrix pencils $\la N_1 + M_1$ and $\la N_2 + M_2$ are said to be {\em strictly equivalent} if there are two invertible complex matrices $P,Q$ such that $P(\la N_1 + M_1)Q=\la N_2 + M_2$ and {\em strictly real equivalent} if $P$ and $Q$ are real. Theorems \ref{thm.1strealCFandKCF} and \ref{thm.2strealCFandKCF} establish, respectively, the relations between the {\em real congruence} canonical forms of $A\in \R^{n\times n}$ in Theorems \ref{main_th} and \ref{second_main_th} and the real Kronecker canonical form {\em under strict real equivalence} of $(A^\top , A)$. More precisely, Theorems \ref{thm.1strealCFandKCF} and \ref{thm.2strealCFandKCF} prove that the real congruence canonical forms fully determine the real Kronecker canonical form but that {\em the real Kronecker canonical form determines the real congruence canonical forms only up to the sign of certain canonical blocks}. Despite this indetermination, we will see throughout the paper that the real Kronecker canonical form under strict real equivalence of $(A^\top , A)$ is a very useful tool in the study of the real congruence canonical forms of $A$, because the class of real strict equivalence transformations is much larger than the class of real congruence transformations. So, it is often easier to prove that two real pairs $(A^\top , A)$ and $(B^\top , B)$ are strictly real equivalent than to prove that the matrices $A$ and $B$ are real congruent.

\section{Notation, basic notions, and basic results} \label{sec.prelim}
In this section, we present some definitions and results that will be used in the rest of the paper. Most of them are well-known or are direct consequences of well-known results. Others are new.

\subsection{Congruence of matrices. Equivalence of matrix pairs}
By $I_k$ we denote the $k\times k$ identity matrix, and $\ii$ denotes the imaginary unit (namely $\ii^2=-1$). The notation $M^\top$ and $M^*$ is used for, respectively, the transpose and the conjugate transpose of the matrix $M$. The {\em$^*$\!cosquare} of an invertible matrix $M$ is the matrix $M^{-*}M$, where $M^{-*}$ denotes the conjugate transpose of the inverse of $M$. By $\C^{m\times n}$ and $\R^{m\times n}$ we denote the sets of $m\times n$ complex and real matrices, respectively.

\begin{definition} \label{def.cong}
Two matrices $A,B\in\C^{n\times n}$ are
\begin{itemize}
\item {\em $^*$congruent} if there is some invertible matrix $S\in\C^{n\times n}$ such that $ SAS^*=B$,
\item {\em congruent} if there is some invertible matrix $S\in\C^{n\times n}$ such that $SAS^\top=B$,
\item {\em real-congruent} if there is some invertible matrix $S\in\R^{n\times n}$ such that $SAS^\top=B$.
\end{itemize}
\end{definition}
\noindent The interesting case for real congruence is when $A,B$ both have real entries.

The following theorem will be fundamental in the proofs of the main results of this paper. It states that if two real matrices are $^*$\!congruent then they are also real-congruent.

\begin{theorem}\label{di_th}{\rm \cite[Th. 1.1]{di02}}
 Let $A,B\in\R^{n\times n}$ be such that $PAP^* = B$,  for some invertible $P\in\C^{n\times n}$. Then, there exists an invertible $Q\in\R^{n\times n}$ such that
$QAQ^\top=B$.
\end{theorem}

The counterpart of Theorem \ref{di_th} for the {\em similarity} of matrices is well-known and can be found in \cite[Theorem 1.3.29]{hj13}, for instance. We remark that the proof of Theorem \ref{di_th} is considerably more involved than the one of \cite[Theorem 1.3.29]{hj13}.

Matrix pairs, or matrix pencils, will play also an important role in this paper. So, we recall some related definitions.
\begin{definition} \label{def.equivpairs}
Let $A,B,C,D \in \C^{m\times n}$. The matrix pairs $(A,B)$ and $(C,D)$ are
\begin{itemize}
    \item {\em strictly equivalent} if there are invertible matrices $R \in \C^{m\times m}$ and $S \in \C^{n\times n}$ such that $RAS = C$ and $RBS = D$,
    \item {\em strictly real-equivalent} if there are invertible matrices $R \in \R^{m\times m}$ and $S \in \R^{n\times n}$ such that $RAS = C$ and $RBS = D$.
\end{itemize}
\end{definition}
\noindent The interesting case for strict real-equivalence is when $A,B,C,D$ have real entries.

A matrix pair $(A,B)$ can also be seen as a matrix pencil $\lambda A + B$, where $\lambda$ is a variable \cite[Ch. XII]{gant}. We will use both views throughout the paper. We emphasize that, in this paper, in the pencil view the first matrix of a pair $(A,B)$ is the leading coefficient of the corresponding pencil and that we use the ``$+$'' sign in the definition of the pencil. We will denote strict equivalence of pairs by $\approx$ and strict real-equivalence by $\stackrel{r}{\approx}$. For simplicity, given matrices $R,A,B,S$ of adequate sizes, we define the product of matrices times matrix pairs as $R(A,B)S := (RAS,RBS)$.

The following result states that if two real pencils are strictly equivalent then they are also strictly real-equivalent. We include the proof since we have not found it in the literature.

\begin{lemma} \label{lemm.strictrealpenc} Let $(A,B) \in \R^{m\times n} \times \R^{m\times n}$ and
 $(C,D) \in \R^{m\times n} \times \R^{m\times n}$ be such that $R (A,B) = (C,D) S$ for some invertible matrices $R \in \C^{m\times m}$ and $S \in \C^{n\times n}$. Then there exist invertible matrices $\widetilde R \in \R^{m\times m}$ and $\widetilde S \in \R^{n\times n}$ such that $\widetilde R (A,B) = (C,D) \widetilde S$.
\end{lemma}

\begin{proof}
Let $R = R_r + \ii R_\ii$ with $R_r, R_\ii \in \R^{m\times m}$ and $S = S_r + \ii S_\ii$ with $S_r, S_\ii \in \R^{n\times n}$, i.e., we express $R$ and $S$ in terms of their real and imaginary parts. Then, $R (A,B) = (C,D) S$ implies $R_r (A,B) = (C,D) S_r$ and $R_\ii (A,B) = (C,D) S_\ii$. So, for any number $\tau$, $(R_r + \tau R_\ii) (A,B) = (C,D) (S_r + \tau S_\ii)$. It only remains to prove that we can choose $\tau_0 \in \R$, such that $\det (R_r + \tau_0 R_\ii) \ne 0$ and $\det (S_r + \tau_0 S_\ii) \ne 0$. For this purpose note that the polynomials in $\tau$, $p(\tau) = \det (R_r + \tau R_\ii)$ and $q(\tau) = \det (S_r + \tau S_\ii)$ are not identically zero since  $p(\ii) \ne 0$ and  $q(\ii) \ne 0$. Moreover, $p(\tau)$ has at most $m$ complex roots and $q(\tau)$ has at most $n$ complex roots. Thus, we can take $\tau_0$ to be equal to any real number which is not a root of $p(\tau) \, q(\tau)$ and $\widetilde R = R_r + \tau_0 R_\ii$ and $\widetilde S = S_r + \tau_0 S_\ii$.
\end{proof}

\subsection{Canonical blocks} The canonical forms considered in this work are direct sums of certain canonical blocks that are described in this subsection. We also investigate some properties of these canonical blocks. As usual, the direct sum of two matrices is defined as $A\oplus B := \left[\begin{smallmatrix}
A & 0 \\ 0 & B \end{smallmatrix} \right]$.

The following matrices were introduced in \cite[p. 1011]{hs06}, for each integer $k\geq1$ and $\mu\in\C$:
\begin{eqnarray}
   J_k(\mu):= \begin{bmatrix}
       \mu &1\\&\ddots&\ddots\\&&\mu&1\\&&&\mu
   \end{bmatrix}_{k\times k},\label{jordanblock}\\
   \Gamma_k:=\left[\begin{array}{c@{\mskip8mu}c@{\mskip8mu}c@{\mskip8mu}c@{\mskip8mu}c@{\mskip8mu}c}0&&&&&(-1)^{k+1}\\[-4pt]
&&&&\iddots&(-1)^k\\[-4pt]&&&-1&\iddots&\\&&1&1&&\\&-1&-1&&&\\1&1&&&&0\end{array}\right]_{k\times k},\label{gamma}\\
    H_{2k}(\mu):=\begin{bmatrix}
    0&I_k\\
     J_k(\mu)&0
    \end{bmatrix}_{2k\times2k}\label{h2k}.
\end{eqnarray}
The matrix in \eqref{jordanblock} is a Jordan block associated with the eigenvalue $\mu$ (see, for instance, \cite[Def. 3.1.1]{hj13}). We highlight that $\Gamma_1=1$ and $J_1(\mu)=\mu$.

Following the notation in \cite[p. 202]{hj13}, for $a,b\in\R$ and $k\geq1$, we define the matrices
\begin{equation}\label{matrixc}
C(a,b):=\begin{bmatrix}
        a&b\\-b&a
    \end{bmatrix} \quad \mbox{and} \quad
    C_{2k}(a,b):=\begin{bmatrix}
        C(a,b)&I_2\\&C(a,b)&\ddots\\&&\ddots&I_2\\&&&C(a,b)
    \end{bmatrix}_{2k\times2k}.
\end{equation}
Note that we write $C_{2k}(a,b)$ instead of $C_k(a,b)$ (as in \cite{hj13}) to highlight that the size of the matrix is $2k\times 2k$, as we have done with $H_{2k}(\mu)$.

We will also need the matrix (again for $a,b\in\R$ and $k\geq1$):
\begin{equation}\label{h4k}
    \widehat H_{4k}(a,b):=\begin{bmatrix}
            0&I_{2k}\\C_{2k}(a,b)&0
        \end{bmatrix}.
\end{equation}

Lemmas \ref{mu_lemma} and \ref{cab_lemma} are slight variants of some identities from \cite[pp. 201-202]{hj13}. They will be used later. Lemma \ref{mu_lemma} follows from a direct computation.

\begin{lemma}\label{mu_lemma}
    Let $\mu=a+\ii b$, with $a,b\in\R$ and $b\neq0$, and $D(\mu) :=\left[\begin{smallmatrix}
       \mu&0\\0&\overline\mu
   \end{smallmatrix}\right]$. Then the unitary matrix $W = \frac{1}{\sqrt{2}}\left[\begin{smallmatrix}
       -\ii&-\ii\\1&-1
   \end{smallmatrix}\right]$ satisfies $W D(\mu) W^* = C(a,b)$.
\end{lemma}

In the proof of the next lemma, as well as in many other parts of the paper, we use the Kronecker product $A \otimes B$ of two matrices and some of its properties. The reader can find a lot of information about it in \cite[Ch. 4]{hj-topics}.

\begin{lemma}\label{cab_lemma}
    Let $\mu=a+\ii b$, with $a,b\in\R$ and $b\neq0$. Then there is a unitary matrix $U\in \C^{2k \times 2k}$ such that
    $
    U\left[\begin{smallmatrix}
        J_k(\mu)&0\\0&J_k(\overline\mu)
    \end{smallmatrix}\right]U^*=C_{2k}(a,b).
    $
    \end{lemma}
\begin{proof}
   According to \cite[p. 201]{hj13}, there is a permutation matrix $P$ such that
   $$
   (I_k\otimes W)P\left[\begin{smallmatrix}
       J_k(\mu)&0\\0&J_k(\overline\mu)
   \end{smallmatrix}\right]P^\top(I_k\otimes W^*)=(I_k\otimes W)\begin{bmatrix}
       D(\mu)&I_2\\&D(\mu)&\ddots\\&&\ddots&I_2\\&&&D(\mu)
   \end{bmatrix}(I_k\otimes W^*)=C_{2k}(a,b),
   $$
   where $D(\mu)$ and $W$ are as in Lemma \ref{mu_lemma}. Setting $U=(I_k\otimes W)P$ we get the result.
\end{proof}

The following tridiagonal matrices were introduced in \cite[eqs. (6) and (7)]{fhs08}, for each integer $k\geq 1$ and $\mu \in \C$:
\begin{equation} \label{eq.blocktri1}
T_k (\mu) := \left[
\begin{array}{ccccc}
0 & 1 & & & 0\\
\mu & 0 & 1 & & \\
& \mu & 0 & \ddots & \\
& & \ddots & \ddots & 1 \\
0 & & & \mu & 0
\end{array}
\right]_{k \times k}  \qquad (T_1 (\mu) = 0),
\end{equation}
\begin{equation} \label{eq.blocktri2}
\widetilde{\Gamma}_k := \left[
\begin{array}{cccccc}
1&1&&&&0 \\
-1&0&1&&& \\
&1&0&1&& \\
&&-1&0&1& \\
&&&1&0& \ddots \\
0&&&&\ddots& \ddots
\end{array} \right]_{k \times k}  \qquad (\widetilde{\Gamma}_1 = 1).
\end{equation}
From $T_k (\mu)$ in \eqref{eq.blocktri1} and the first matrix in \eqref{matrixc}, we define, for $a,b \in \R$ and $k \geq 1$, the matrix:
\begin{equation} \label{eq.blocktri3}
\widehat{T}_{2k} (a,b) := \left[
\begin{array}{ccccc}
0 & I_2 & & & 0\\
C(a,b) & 0 & I_2 & & \\
& C(a,b) & 0 & \ddots & \\
& & \ddots & \ddots & I_2 \\
0 & & & C(a,b) & 0
\end{array}
\right]_{2k \times 2k} \qquad (\widehat{T}_{2} (a,b) = 0_{2\times 2}).
\end{equation}
We warn the reader that in the main results of this paper where the matrix in \eqref{eq.blocktri3} plays a key role, i.e., Theorems \ref{second_main_th}, \ref{thm.2strealCFandKCF}, and \ref{thm.relation2}, the parameter $k$ defining the number of $2\times 2$ blocks is an even number and so the matrix appears written as $\widehat{T}_{4k} (a,b)$.

The next lemma relates the matrices \eqref{eq.blocktri1} and \eqref{eq.blocktri3}. It resembles Lemma \ref{cab_lemma}.

\begin{lemma} \label{lemm.reltridiablocks} Let $\mu = a + \ii b$, with $a,b \in \R$ and $b \ne 0$. Then there is a unitary matrix $V \in \C^{2k \times 2k}$ such that $V \left[ \begin{smallmatrix}
T_k (\mu) & 0 \\ 0 & T_k (\overline \mu)
\end{smallmatrix}\right] V^*= \widehat{T}_{2k} (a,b)$.
\end{lemma}
\begin{proof} Let $P \in  \C^{2k \times 2k}$ be the permutation matrix which corresponds to permuting the rows \break$\begin{bmatrix}
1, &2, & \ldots ,& k, & k+1, & k+2, & \ldots , & 2k
\end{bmatrix}$ of $I_{2k}$ as follows
\begin{equation*}
\begin{bmatrix}
1, & k+1, &2 & k+2, & 3, & k +3, & \ldots , & k, & 2k
\end{bmatrix},
\end{equation*}
i.e., the first $k$ rows of $I_{2k}$ are in the odd positions in $P$ and the last $k$ rows of $I_{2k}$ are in the even positions, preserving in both cases the relative order.
If $W$ and $D(\mu)$ are the matrices in Lemma \ref{mu_lemma}, then
\begin{align*}
(I_k \otimes W) P \left[ \begin{array}{cc}
T_k (\mu) & 0 \\ 0 & T_k (\overline \mu)
\end{array}\right] P^\top (I_k \otimes W^*) & = (I_k \otimes W)
\left[
\begin{array}{ccccc}
0 & I_2 & & & 0\\
D(\mu) & 0 & I_2 & & \\
& D(\mu) & 0 & \ddots & \\
& & \ddots & \ddots & I_2 \\
0 & & & D(\mu) & 0
\end{array}
\right]
(I_k \otimes W^*) \\ & = \widehat{T}_{2k} (a,b).
\end{align*}
Setting $V= (I_k \otimes W) P$ concludes the proof.
\end{proof}

We will also use the following result.

\begin{lemma}\label{similar_lemma}
    If $A$ and $B$ are square $n\times n$ similar matrices, then $\left[\begin{smallmatrix}
        0&I_n\\A&0
    \end{smallmatrix}\right]$ and $\left[\begin{smallmatrix}
        0&I_n\\B&0
    \end{smallmatrix}\right]$ are congruent and $^*$\!congruent.
\end{lemma}
\begin{proof}
    Let $S$ be nonsingular such that $SAS^{-1}=B$. Then $M\left[\begin{smallmatrix}
        0&I_n\\A&0
    \end{smallmatrix}\right]M^\star=\left[\begin{smallmatrix}
        0&I_n\\B&0
    \end{smallmatrix}\right]$, where $M=\left[\begin{smallmatrix}
        S^{-\star}&0\\0&S
    \end{smallmatrix}\right]$, and with $\star$ being either $\top$ or $*$.
\end{proof}

Lemma \ref{lemm.problocksHS} states two properties of the matrices in \eqref{gamma} and \eqref{h2k} which will be often used. The proof is omitted because it is essentially provided in \cite[p. 1016]{hs06}, with $H_{2k} (\mu)^*$ replaced by $H_{2k} (\mu)^\top$.

\begin{lemma} \label{lemm.problocksHS} Let $\Gamma_k$ and $H_{2k} (\mu)$ be the matrices defined in \eqref{gamma} and \eqref{h2k}. Then
\begin{enumerate}
    \item  $\Gamma_k^{-\top} \Gamma_k$ is similar to $J_k ((-1)^{k+1})$, and
\item $H_{2k} (\mu)^{-*} H_{2k} (\mu)$ is similar to $\left[ \begin{smallmatrix}
J_k (\mu) & 0 \\ 0 & J_k (1/\overline \mu)
\end{smallmatrix}\right]$.\label{h2k-cosquare}
\end{enumerate}
\end{lemma}

The key property of the Kronecker product in Lemma \ref{lemm.commukron} will be applied in several proofs. It is a particular case of \cite[Cor. 4.3.10]{hj-topics}.

\begin{lemma} \label{lemm.commukron} Let $A \in \C^{n\times n}$ and $B \in \C^{p\times p}$. Then there is a permutation matrix $P(n,p) \in \C^{np\times np}$, depending only on the dimensions $n$ and $p$, such that
$A\otimes B = P(n,p) \, (B\otimes A) \,P(n,p)^\top$.
\end{lemma}

\subsection{Kronecker and real-Kronecker canonical forms of matrix pairs} We revise in this section the Kronecker Canonical Form (KCF) of a complex matrix pair under strict equivalence and the real Kronecker Canonical Form (real-KCF) of a real matrix pair under strict real-equivalence. The reason is that we will relate the canonical forms under $^*$\!congruence of a matrix $A\in \C^{n\times n}$ (resp. under real-congruence of a matrix $A\in \R^{n\times n}$) with the KCF of the complex pair $(A^* , A)$ (resp. with the real-KCF of the real pair $(A^\top , A)$). We will need the following two additional matrices for defining the KCF:
\begin{equation} \label{eq.singblocks}
F_k := \begin{bmatrix}
1 & 0 & & 0 \\
 & \ddots & \ddots & \\
 0 & & 1 & 0
\end{bmatrix}_{k \times (k+1)} \qquad \mbox{and} \quad
G_k := \begin{bmatrix}
0 & 1 & & 0 \\
 & \ddots & \ddots & \\
 0 & & 0 & 1
\end{bmatrix}_{k \times (k+1)} ,
\end{equation}
where $F_0 = G_0$ is the $0\times 1$ matrix. The direct sum of two matrix pairs is defined in a natural way as $(A,B) \oplus (C,D) = (A\oplus C, B\oplus D)$.

\begin{theorem} \label{thm.KCF} {\rm (KCF, \cite[Ch. XII, Theorem 5]{gant})} Each matrix pair $(A,B) \in \C^{m\times n} \times \C^{m\times n}$ is strictly equivalent to a direct sum, uniquely determined up to permutation of summands, of canonical pairs of the following four types
\begin{center}
    {\renewcommand{\arraystretch}{1.6}
     \renewcommand{\tabcolsep}{0.3cm}
    \begin{tabular}{|c|c|}\hline
         Regular pairs for finite eigenvalues & $(I_k, J_k(\mu ))$ with $\mu \in \C$ \\\hline
         Regular pairs for infinite eigenvalues & $(J_k(0),I_k)$  \\\hline
         Right singular pairs & $(F_k,G_k)$ \\\hline
         Left singular pairs & $(F_k^\top,G_k^\top)$ \\\hline
    \end{tabular}}
    \end{center}
\end{theorem}

The direct sum asserted in Theorem \ref{thm.KCF} is (up to permutation of its direct
 summands) the KCF of $(A,B)$ and it will be denoted by KCF$(A,B)$.

\begin{theorem} \label{thm.realKCF} {\rm (real-KCF, \cite[Theorem 3.2]{lr-siamrev-05})} Each matrix pair $(A,B) \in \R^{m\times n} \times \R^{m\times n}$ is strictly real-equivalent to a direct sum, uniquely determined up to permutation of summands, of real canonical pairs of the following five types
\begin{center}
    {\renewcommand{\arraystretch}{1.6}
     \renewcommand{\tabcolsep}{0.25cm}
    \begin{tabular}{|c|c|}\hline
         Regular pairs for finite real eigenvalues & $(I_k, J_k(\mu ))$ with $\mu \in \R$ \\\hline
         Regular pairs for infinite eigenvalues & $(J_k(0),I_k)$  \\\hline
         Regular pairs for finite complex-conjugate eigenvalues & $(I_{2k}, C_{2k} (a,b))$ with $a,b \in \R, b>0$ \\\hline
         Right singular pairs & $(F_k,G_k)$ \\\hline
         Left singular pairs & $(F_k^\top,G_k^\top)$ \\\hline
    \end{tabular}}
    \end{center}
\end{theorem}

The direct sum asserted in Theorem \ref{thm.realKCF} is (up to permutation of its direct
summands) the real-KCF of $(A,B)$ and it will be denoted by real-KCF$(A,B)$.

\subsection{Horn-Sergeichuk canonical form for $^*$congruence}
Due to its relevance in this paper, we reproduce here the canonical form for $^*$\!congruence of a matrix $A \in \C^{n\times n}$ provided in \cite[Theorem 1.1 (b)]{hs06}. Moreover, we relate such canonical form with the KCF of $(A^*,A)$ under strict equivalence. We call the direct sum asserted in Theorem \ref{hs_th} (a) the {\em Horn-Sergeichuk canonical form of $A$} and it will be denoted as HSCF$(A)$, where it must be understood that HSCF$(A)$ is defined up to permutation of the direct summands. We emphasize that Theorem \ref{hs_th} (b) establishes that KCF$(A^*,A)$ determines HSCF$(A)$ up to the signs of a particular type of canonical blocks. This may be useful in determining HSCF$(A)$ because the class of strict equivalence transformations is larger than the class of $^*$\!congruence transformations.

\begin{theorem} \label{hs_th} \phantom{kk}
\begin{itemize}
    \item[\rm (a)] {\rm \cite[Th. 1.1 (b)]{hs06}}
    Each square complex matrix $A$ is $^*$\!congruent to a direct sum, uniquely determined up to permutation of summands, of canonical matrices of the following three types
    \begin{center}
    {\renewcommand{\arraystretch}{1.6}
     \renewcommand{\tabcolsep}{0.3cm}
    \begin{tabular}{|c|c|c|}\hline
    Name&Block&Conditions\\\hline\hline
      {\rm  Type 0} & $J_k(0)$ & --\\\hline
      {\rm  Type I} & $\mu\Gamma_k$& $\mu\in \C$, $|\mu|=1$\\\hline
      {\rm  Type II} &$H_{2k}(\mu)$& $\mu\in \C$, $|\mu|>1$\\\hline
    \end{tabular}}
    \end{center}

    \item[\rm (b)] The direct sum asserted in (a) determines the {\rm KCF}$(A^*,A)$ under strict equivalence uniquely up to permutation of its direct summands. Conversely, the {\rm KCF}$(A^*,A)$ under strict equivalence determines the direct sum asserted in (a) uniquely up to permutation of summands and multiplication of any direct summand of Type I by $-1$. For any direct summand $B$ of Types I, II, or III in part (a), the {\rm KCF}$(B^*,B)$ under strict equivalence is given in the following table:
     \begin{center}
    {\renewcommand{\arraystretch}{2.2}
     \renewcommand{\tabcolsep}{0.25cm}
    \begin{tabular}{|c|c|}\hline
         Block $B$ in {\rm HSCF} & {\rm KCF}$(B^*,B)$\\ \hline \hline
         $J_k(0)$ &
         \begin{tabular}{ll}
          $(F_\ell , G_\ell) \oplus (F_\ell^\top , G_\ell^\top)$   &  if $k=2\ell +1$ \\
          $(J_\ell (0), I_\ell ) \oplus (I_\ell \, J_\ell (0))$    &  if $k=2\ell$
         \end{tabular}
         \\ \hline
         $\mu\Gamma_k$, \; $\mu\in \C$, $|\mu|=1$ &  $\displaystyle  \left(I_k \, , \, J_k \left( (-1)^{k+1} \, \mu^2 \right) \, \right)$ \\[0.2cm] \hline
         $H_{2k}(\mu)$, \; $\mu\in \C$, $|\mu|>1$ & $\displaystyle (I_k \, , \, J_k (\mu)\,)  \oplus  (I_k \, , \, J_k \left( 1/\overline{\mu} \right) \, )$ \\[0.2cm] \hline
    \end{tabular}}
    \end{center}
\end{itemize}
\end{theorem}
\noindent{\em Proof of }(b). Observe that if $A = P (B_1 \oplus \cdots \oplus  B_q) P^*$ is the $^*$\!congruence asserted in part (a), where each $B_i$ is a canonical matrix of Type 0, I, or II, then $(A^*,A) = P \left( (B_1^*,B_1) \oplus \cdots \oplus  (B_q^*,B_q)  \right) P^*$. So, $(A^*,A) \approx (B_1^*,B_1) \oplus \cdots \oplus  (B_q^*,B_q)$ and ${\rm KCF}(A^*,A) = {\rm KCF}(B_1^*,B_1) \oplus \cdots \oplus  {\rm KCF}(B_q^*,B_q)$. Thus, HSCF$(A)$ and the table in part (b) determine completely KCF$(A^*,A)$. On the other hand, from the table in part (b), it is obvious that KCF$(A^*,A)$ determines HSCF$(A)$ only up to multiplication of any direct summand of Type I by $-1$. Therefore, it only remains to prove the table in part (b).

The KCF of $(J_k (0)^* , J_k(0))$ follows from \cite[Theorem 1.2 (b)]{fhs08} by taking $\lambda =0$ in that reference. It can be also easily deduced via standard arguments of matrix pencils. For KCF$((\mu \Gamma_k)^* , \mu \Gamma_k)$ with $|\mu|=1$, we proceed as follows
\begin{align*}
((\mu \Gamma_k)^* , \mu \Gamma_k) ) & \approx (I_k , (\mu \Gamma_k)^{-*} \mu \Gamma_k) ) = \left(I_k , \frac{\mu}{\overline \mu} \, \Gamma_k^{-\top} \Gamma_k) \right) \approx \left(I_k , J_k \left( (-1)^{k+1} \frac{\mu}{\overline \mu}  \right) \right) \\ & =  \left(I_k , J_k \left( (-1)^{k+1} \mu^2  \right) \right),
\end{align*}
where the last strict equivalence follows from Lemma \ref{lemm.problocksHS} and the last equality from the fact that $\mu^2 = \mu / \overline \mu$ if $|\mu| =1$.  Analogously, for KCF$( H_{2k} (\mu)^* , H_{2k} (\mu) )$, observe that
$$
( H_{2k} (\mu)^* , H_{2k} (\mu) ) \approx (I_{2k} ,  H_{2k} (\mu)^{-*} H_{2k} (\mu) ) \approx (I_k \, , \, J_k (\mu)\, ) \oplus  (I_k \, , \, J_k \left( 1/\overline{\mu} \right) \, ) ,
$$
where the last strict equivalence follows from Lemma \ref{lemm.problocksHS}.
\qed

\begin{remark}\label{mu_rem} We emphasize that the condition $|\mu|>1$ in the Type II blocks in Theorem {\rm\ref{hs_th}} (a) can be replaced by $0 <|\mu|<1$. To see this note that $\left[\begin{smallmatrix}
    0&J_k(\mu)^{-1}\\I&0
\end{smallmatrix}\right]\left[\begin{smallmatrix}
    0&I\\J_k(\mu)&0
\end{smallmatrix}\right]\left[\begin{smallmatrix}
  0&J_k(\mu)^{-1}\\I&0
\end{smallmatrix}\right]^*=\left[\begin{smallmatrix}
    0&I\\J_k(\mu)^{-*}&0
\end{smallmatrix}\right]$ and, since $J_k(\mu)^{-*}$ is similar to $J_k(1/\overline\mu)$, we can use Lemma {\rm\ref{similar_lemma}} to replace $\mu$ by $1/\overline\mu$.
\end{remark}

\subsection{Futorny-Horn-Sergeichuk tridiagonal canonical form for $^*$congruence}
Another interesting canonical form for $^*$\!congruence of a matrix $A \in \C^{n\times n}$ is the tridiagonal one introduced in \cite[Theorem 1.2]{fhs08}, which requires fewer canonical blocks than the HSCF described in Theorem \ref{hs_th}. We reproduce here \cite[Theorem 1.2]{fhs08}, since we will also develop a real counterpart of this canonical form. We remark that \cite[Theorem 1.2]{fhs08} is valid for any algebraically closed field with a nonidentity involution, but, taking into account the purposes of this work, we will state it over $\C$. The direct sum asserted in Theorem \ref{thm.fhs} (a) will be called the {\em Futorny-Horn-Sergeichuk canonical form of $A$} and will be denoted by FHSCF$(A)$, which is defined up to permutation of the direct summands. As in the case of Theorem \ref{hs_th}, the FHSCF$(A)$ will be related to KCF$(A^*,A)$, although in this case such a relationship was already established in \cite[Theorem 1.2]{fhs08}. Also in this case, KCF$(A^*,A)$ determines FHSCF$(A)$ up to the signs of a particular type of canonical blocks. It is natural to wonder about the precise relationship between the HSCF and the FHSCF. We postpone the answer to this question to Section \ref{sec.relHS-FHS}, as it requires careful analysis. We warn the reader that we have used $\mu^2 = \overline{\mu}^{-1} \mu$ if $|\mu| = 1$, and $(J_k (\overline \alpha) , I_k) \approx (I_k, J_k (1/\overline \alpha))$ if $\alpha \ne 0$ in Theorem \ref{thm.fhs} (b).

\begin{theorem} {\rm \cite[Theorem 1.2]{fhs08}} \label{thm.fhs}
\phantom{kk}
\begin{itemize}
    \item[\rm (a)]
    Each square complex matrix $A$ is $^*$\!congruent to a direct sum, uniquely determined up to permutation of summands, of tridiagonal canonical matrices of the following two types
    \begin{center}
    {\renewcommand{\arraystretch}{1.6}
     \renewcommand{\tabcolsep}{0.3cm}
    \begin{tabular}{|c|l|l|}\hline Name&Block&Conditions\\\hline\hline
         {\rm Type Tri-I} & $T_k(\mu)$ &
         \begin{tabular}{l}
         $\mu \in \C$, $|\mu| \ne 1$, \\[-0.2cm]
         each nonzero $\mu$ is determined \\[-0.2cm]
         up to replacement by $\overline{\mu}^{-1}$,\\[-0.2cm]
         $\mu = 0$ if $k$ is odd
         \end{tabular}
         \\[0.2cm] \hline
        {\rm  Type Tri-II} & $\mu\widetilde{\Gamma}_k$ & \begin{tabular}{l} $\mu \in \C$, $|\mu|=1$ \end{tabular}
         \\[0.2cm] \hline
    \end{tabular}}
    \end{center}
    \item[\rm (b)] The direct sum asserted in (a) determines the {\rm KCF}$(A^*,A)$ under strict equivalence uniquely up to permutation of its direct summands. Conversely, the {\rm KCF}$(A^*,A)$ under strict equivalence determines the direct sum asserted in (a) uniquely up to permutation of summands and multiplication of any direct summand of Type Tri-II by $-1$. For any direct summand $B$ of Types Tri-I or Tri-II in part (a), the {\rm KCF}$(B^*,B)$ under strict equivalence is given in the following table:
     \begin{center}
    {\renewcommand{\arraystretch}{2.2}
     \renewcommand{\tabcolsep}{0.25cm}
    \begin{tabular}{|c|c|}\hline
          Block $B$ in {\rm FHSCF} & {\rm KCF}$(B^*,B)$\\ \hline \hline
         $T_k(\mu)$, \; \begin{tabular}{l} $\mu\in \C$, $|\mu| \ne 1$, \\[-0.2cm]
         $\mu = 0$ if $k$ is odd
         \end{tabular}
         &
         \begin{tabular}{ll}
          $(F_\ell , G_\ell) \oplus (F_\ell^\top , G_\ell^\top)$   &  if $k=2\ell +1$ \\
          $(J_\ell (0), I_\ell ) \oplus (I_\ell , J_\ell (0))$    &  if $k=2\ell$ and $\mu =0$ \\
          $\displaystyle (I_\ell \, , \, J_\ell (\mu)\,)  \oplus  (I_\ell \, , \, J_\ell \left( 1/\overline{\mu} \right) \, )$  &  if $k=2\ell$ and $\mu \ne 0$ \\[0.2cm]
         \end{tabular}
         \\ \hline
         $\mu \widetilde{\Gamma}_k$, \; $\mu\in \C$, $|\mu|=1$ &  $\displaystyle  \left(I_k \, , \, J_k \left( (-1)^{k+1} \, \mu^2 \right) \, \right)$ \\[0.2cm] \hline
    \end{tabular}}
    \end{center}
\end{itemize}
\end{theorem}

\section{Relation between Horn-Sergeichuk and Futorny-Horn- Sergeichuk canonical forms} \label{sec.relHS-FHS}

We present in Theorem \ref{thm.relation1} the precise $^*$congruence relations between the canonical blocks in the HSCF of Theorem \ref{hs_th} (a) and those in the FHSCF of Theorem \ref{thm.fhs} (a). For that purpose, we need the auxiliary Lemma \ref{equivalencegammas.th}, which shows that $\Gamma_k$ is real-congruent to either $\widetilde\Gamma_k$, when $k\equiv1,2$ (mod $4$), or $-\widetilde\Gamma_k$, when $k\equiv0,3$ (mod $4$). Moreover, it also provides the explicit real-congruence between these two matrices, which is the product of a permutation matrix times a diagonal signature matrix and, so, is orthogonal. The proof of Lemma \ref{equivalencegammas.th} is postponed to Appendix \ref{append.proof-lemma}.

\begin{lemma}\label{equivalencegammas.th}
    For any $k\geq1$, let $P_k \in \R^{k\times k}$ be the permutation matrix which corresponds to permuting the columns $\begin{bmatrix} 1, &\ldots,&k\end{bmatrix}$ of $I_k$ as follows
    \begin{eqnarray}
    \begin{bmatrix}
        \frac{k}{2}+1,&\frac{k}{2},&\frac{k}{2}+2,&\frac{k}{2}-1,&\hdots,&k-1,&2,&k,&1
    \end{bmatrix}  &\mbox{if $k$ is even, and}\\
    \begin{bmatrix}
        \frac{k+1}{2},&\frac{k+3}{2},&\frac{k-1}{2},&\frac{k+5}{2},&\frac{k-3}{2},&\hdots,&k-1,&2,&k,&1
    \end{bmatrix}&\mbox{if $k$ is odd}.
    \end{eqnarray}
    Let also $S_k:=\diag(s_1,\hdots,s_k)$, where
    $$
    s_i=\left\{\begin{array}{cc}
         -1&\mbox{if $i\equiv3$ {\rm(mod 4)}},  \\
         1&\mbox{otherwise.}
    \end{array}\right.
    $$
    Then
    \begin{equation*}\label{congruencegamma}
        (P_kS_k)^\top\Gamma_k(P_kS_k)=\left\{\begin{array}{cc}
             \widetilde\Gamma_k&\mbox{if $k\equiv1,2$ {\rm(mod 4)}},  \\
             -\widetilde\Gamma_k&\mbox{if $k\equiv0,3$ {\rm (mod 4)}.}
        \end{array}\right.
    \end{equation*}
\end{lemma}

Theorem \ref{thm.relation1} is a direct corollary of Lemma \ref{equivalencegammas.th} and part (b) in Theorems \ref{hs_th} and \ref{thm.fhs}, so the proof is omitted.

\begin{theorem} \label{thm.relation1} Let $J_k(0)$, $\Gamma_k$, $H_{2k} (\mu)$, $T_k(\mu)$, and $\widetilde \Gamma_k$ be the matrices in \eqref{jordanblock}, \eqref{gamma}, \eqref{h2k}, \eqref{eq.blocktri1}, and \eqref{eq.blocktri2}, respectively. Then the $^*$\!congruences described in the following table hold.
 \begin{center}
    {\renewcommand{\arraystretch}{2.2}
     \renewcommand{\tabcolsep}{0.25cm}
    \begin{tabular}{|c|c|}\hline
          Block $B$ in {\rm HSCF} & is $^*$\!congruent to block $C$ in {\rm FHSCF}\\ \hline \hline
         $J_k(0)$ & $T_k (0)$
         \\ \hline
         $\mu\Gamma_k$, \; $\mu\in \C$, $|\mu|=1$ &  $\begin{array}{cc}
             \mu \widetilde\Gamma_k&\mbox{if $k\equiv1,2$ {\rm(mod 4)}}  \\[-0.15cm]
             -\mu \widetilde\Gamma_k&\mbox{if $k\equiv0,3$ {\rm (mod 4)}}
        \end{array}$ \\[0.2cm] \hline
         $H_{2k}(\mu)$, \; $\mu\in \C$, $|\mu|>1$ & $T_{2k} (\mu)$ \\[0.2cm] \hline
    \end{tabular}}
    \end{center}
\end{theorem}

\section{First canonical form of real matrices for real-congruence} \label{sec.firstcan}

Theorem \ref{main_th} establishes a real counterpart of the HSCF presented in Theorem \ref{hs_th} (a). It is one of the two main results of this paper. Theorem \ref{main_th} will be complemented in Theorem \ref{thm.1strealCFandKCF} with the relation between the first canonical form of $A\in \R^{n\times n}$ for real-congruence and the real-KCF of $(A^\top,A)$.

\begin{theorem}\label{main_th}
    Each square matrix $A\in\R^{n\times n}$ is real-congruent to a direct sum, uniquely determined up to permutation of summands, of canonical matrices of the following four types
 \begin{center}
    {\renewcommand{\arraystretch}{1.8}
     \renewcommand{\tabcolsep}{0.3cm}
    \begin{tabular}{|c|l|l|}\hline Name&Block&\phantom{aaaaaaaa}Conditions\\\hline\hline
         {\rm Type (i)} & $J_k(0)$ &\phantom{aaaaaaaaaaa}-- \\\hline
         {\rm Type (ii)} & $\Gamma_k\otimes N$& \begin{tabular}{l} $N = \pm1$ or \\[-0.2cm]
         $N = C(a,b)$, with $a,b\in\R$, $a^2+b^2=1$, and $b>0$ \end{tabular} \\\hline
         {\rm Type (iii)} & $H_{2k}(a)$& \hspace{0.3cm} $a\in\R$, $0<|a|<1$   \\\hline
         {\rm Type (iv)} & $\widehat H_{4k}(a,b)$&\hspace{0.05cm} $a,b\in\R$, $a^2+b^2<1$, and $b>0$  \\\hline
    \end{tabular}}
    \end{center}
In the matrices of {\rm Type (ii)}, $\Gamma_k\otimes N$ can be replaced by $N \otimes \Gamma_k$. In the matrices of Type {\rm  (iii)}, the condition $0<|a|<1$ can be replaced by $|a| >1$ or, in other words, each $a$ is determined up to replacement by $1/a$. In the matrices of Type {\rm  (iv)}, the condition $a^2 + b^2 <1$ can be replaced by $a^2 + b^2 >1$ or, in other words, each pair $(a,b)$ is determined up to replacement by $\left( a/(a^2 + b^2) \, , \,  b/(a^2 + b^2) \right)$.
\end{theorem}
\begin{proof} According to Theorem \ref{hs_th} (a), there is a nonsingular matrix $P\in \C^{n\times n}$ such that $A=P \, C_A P^*$, where $C_A$ is a direct sum of canonical matrices of the types $J_k(0),\mu\Gamma_k$ with $|\mu|=1$, and $H_{2k}(\mu)$ with $|\mu|>1$ (which can be replaced by $0 <|\mu|<1$, see Remark \ref{mu_rem}), where $J_k(0),\Gamma_k$, and $H_{2k}(\mu)$ are as in \eqref{jordanblock}, \eqref{gamma}, and \eqref{h2k}. By taking conjugates in $A=P\, C_AP^*$ and using that $A$ is real, we get $A=\overline P\, \overline C_A P^\top$. This means that $\overline C_A$ is $^*$\!congruent to $A$. Since $\overline C_A$ is also a direct sum of canonical matrices $J_k(0),\widetilde\mu \, \Gamma_k$ with $|\widetilde\mu|=1$, and $H_{2k}(\widetilde\mu)$ with $|\widetilde\mu|>1$, the uniqueness (up to permutation of summands) of the direct sum in Theorem \ref{hs_th} (a) implies that, for $\mu\in\C\setminus\R$, if a block $\mu\Gamma_k$ appears in $C_A$, the block $\overline\mu\,\Gamma_k$ must appear as well. In other words, the blocks $\mu\Gamma_k,\overline\mu\,\Gamma_k$, with $\mu\in\C\setminus\R$ are paired up in $C_A$. The same happens with the blocks $H_{2k}(\mu)$ and $H_{2k}(\overline\mu)$, with $\mu\in\C\setminus\R$.

Now, we divide the proof in several steps. In Step 1 we see that each pair $\mu\Gamma_k \oplus \overline\mu\,\Gamma_k$, with $\mu\in\C\setminus\R$, is $^*$\!congruent to a real block of Type (ii) with $N = C(a,b)$ in the statement, whereas in Step 2 we do the same with the pairs of blocks $H_{2k}(\mu) \oplus H_{2k}(\overline\mu)$, with $\mu\in\C\setminus\R$, and blocks of Type (iv) in the statement. In Step 3 we show that then there is a real congruence leading $A$ to the direct sum of the blocks produced in Steps 1--2, together with the real blocks of Types (i), (ii) with $N= \pm 1$, and (iii) that were already in $C_A$. Finally, in Step 4 we prove that this final direct sum is unique up to permutation of its direct summands.

\medskip

{\bf Step 1}: Let $\mu=a+\ii b$, with $|\mu|=1$ (so $a^2+b^2=1$) and $b\neq0$. Then
$$
\begin{bmatrix}
    \mu\Gamma_k&0\\0&\overline\mu\,\Gamma_k
\end{bmatrix}=\begin{bmatrix}
    \mu&0\\0&\overline\mu
\end{bmatrix}\otimes\Gamma_k=\Pi\left(\Gamma_k\otimes\begin{bmatrix}
    \mu&0\\0&\overline\mu
\end{bmatrix}\right)\Pi^\top=\Pi(I_k\otimes W)^*\left(\Gamma_k\otimes\begin{bmatrix}
    a&b\\-b&a
\end{bmatrix}\right)(I_k\otimes W)\Pi^\top,
$$
where $\Pi$ is the permutation matrix in Lemma \ref{lemm.commukron}, and $W$ is as in Lemma \ref{mu_lemma}. As a consequence, $\left[\begin{smallmatrix}
    \mu\Gamma_k&0\\0&\overline\mu\,\Gamma_k
\end{smallmatrix}\right]$ is (unitarily) $^*$\!congruent to $\Gamma_k\otimes C(a,b)$.

If $b<0$ then $\left[\begin{smallmatrix}
    0&1\\1&0
\end{smallmatrix}\right]C(a,b)\left[\begin{smallmatrix}
    0&1\\1&0
\end{smallmatrix}\right]=C(a,-b)$, so $C(a,-b)$ is $^*$\!congruent to $C(a,b)$. Therefore, we can restrict ourselves to $b>0$.

Hence every block $\left[\begin{smallmatrix}
    \mu\Gamma_k&0\\0&\overline\mu\,\Gamma_k
\end{smallmatrix}\right]$ is $^*$\!congruent to a block of Type (ii) in the statement, with $N=C(a,b)$ and $b>0$. Lemma \ref{lemm.commukron} proves also that $\Gamma_k \otimes N$ can be replaced by $N \otimes \Gamma_k$ in any direct summand of Type (ii).

\medskip

{\bf Step 2}: Let $\mu=a+\ii b$ with $b>0$. Then, the block $\left[\begin{smallmatrix}
    H_{2k}(\mu)&0\\0&H_{2k}(\overline\mu)
\end{smallmatrix}\right]$ is $^*$\!congruent to $\widehat H_{4k}(a,b)$. To see this, let $U$ be the unitary matrix in Lemma \ref{cab_lemma} and let $P$ be the $4\times 4$ block permutation matrix of size $4k\times 4k$ which exchanges the second and third block rows (or columns). Then
\begin{equation*}\label{h2kperm}
\begin{array}{l}
\left(\begin{bmatrix}
    U&0\\0&U
\end{bmatrix}P\right)\begin{bmatrix}
     H_{2k}(\mu)&0\\0&H_{2k}(\overline\mu)
\end{bmatrix}\left( \begin{bmatrix}
    U&0\\0&U
\end{bmatrix}P\right)^* \\[0.5cm]  \phantom{aaa} =
\begin{bmatrix}
    U&0\\0&U
\end{bmatrix}\begin{bmatrix}
    I_k&0&0&0\\0&0&I_k&0\\0&I_k&0&0\\0&0&0&I_k
\end{bmatrix}\begin{bmatrix}
    0&I_k&0&0\\J_k(\mu)&0&0&0\\0&0&0&I_k\\0&0&J_k(\overline\mu)&0
\end{bmatrix}\begin{bmatrix}
    I_k&0&0&0\\0&0&I_k&0\\0&I_k&0&0\\0&0&0&I_k
\end{bmatrix}\begin{bmatrix}
    U&0\\0&U
\end{bmatrix}^*\\[0.9cm] \phantom{aaa} =\begin{bmatrix}
    U&0\\0&U
\end{bmatrix}\left[\begin{array}{cc|cc}
    0&0&I_k&0\\0&0&0&I_k\\\hline J_k(\mu)&0&0&0\\0&J_k(\overline\mu)&0&0
\end{array}\right]\begin{bmatrix}
    U&0\\0&U
\end{bmatrix}^*=
    \begin{bmatrix}
        0&I_{2k}\\C_{2k}(a,b)&0
    \end{bmatrix}=\widehat H_{4k}(a,b),
\end{array}
\end{equation*}
where, to get the last-but-one identity, we use that $U$ is unitary (namely, $UU^*=I$).

Moreover, note that, if $b<0$, we can consider $\widehat H_{4k}(a,-b)$, since $\widehat H_{4k}(a,b)$ is real-congruent to $\widehat H_{4k}(a,-b)$. To see this, set $\Delta_2:=\left[\begin{smallmatrix}
    0&1\\1&0
\end{smallmatrix}\right]$, and then $(I_{2k}\otimes\Delta_2)\widehat H_{4k}(a,b)(I_{2k}\otimes\Delta_2)^\top=\widehat H_{2k}(a,-b)$.

Finally, according to Remark \ref{mu_rem}, $H_{2k} (\mu)$ with $|\mu| >1$ can be replaced by $H_{2k} (\mu)$ with $0<|\mu| <1$. In the first case, the argument above leads to $\widehat H_{4k}(a,b)$ with $a^2 + b^2 >1$ and in the second to $\widehat H_{4k}(a,b)$ with $a^2 + b^2 <1$.

\medskip

{\bf Step 3}: As a consequence of Steps 1--2, the matrix $A$ is $^*$\!congruent to a direct sum of blocks of Types (i)--(iv) in the statement, that we denote by $C^r_A$. Since both $A$ and $C^r_A$ have real entries, Theorem \ref{di_th} guarantees that  they are also real-congruent.

\medskip

{\bf Step 4 (uniqueness)}: Assume that there are two different direct sums of blocks of Types (i)--(iv), denoted by $C_1$ and $C_2$, which are real-congruent. Then, they are also $^*$\!congruent. But, as we have seen before, blocks of Type (i)--(iv) are $^*$\!congruent to a direct sum of blocks of Types 0, I, and II in Theorem \ref{hs_th}. More precisely, blocks of Type-(i) correspond to Type-0 blocks; blocks of Type-(ii) correspond to either blocks of Type-I with $\mu=\pm1$ or to a direct sum of Type-I blocks when $\mu =a + \ii b \neq\pm1$; blocks of Type-(iii) are particular cases of Type-II blocks; and blocks of Type-(iv) are $^*$\!congruent to a direct sum of Type-II blocks. Note also that the Type-I blocks corresponding to different blocks of Type-(ii) are not $^*$\!congruent to each other, since they correspond to different complex numbers $\mu$ with $|\mu|=1$, and the same happens with the Type-II blocks associated with different blocks of Types (iii) and (iv). Let us denote by $\widehat C_1$ and $\widehat C_2$ the direct sum corresponding to the Type 0, I, and II blocks associated with the blocks of $C_1$ and $C_2$, respectively (in the same order). By Theorem \ref{hs_th}, the blocks in $\widehat C_2$ are a permutation of the blocks in $\widehat C_1$ and, then, the blocks in $C_2$ are also a permutation of the blocks in $C_1$.
\end{proof}


\subsection{Relation of the first canonical form for real-congruence of $A$ and the real-KCF of $(A^\top,A)$}
Theorem \ref{thm.1strealCFandKCF} establishes that real-KCF$(A^\top , A)$ determines the canonical form in Theorem \ref{main_th} up to the signs of some parameters in the blocks of Type (ii).
\begin{theorem} \label{thm.1strealCFandKCF}
Let $A \in \R^{n\times n}$. The direct sum asserted in Theorem {\rm\ref{main_th}} determines {\rm real-KCF}$(A^\top, A)$ under strict real-equivalence uniquely up to permutation of its direct summands. Conversely, the {\rm real-KCF}$(A^\top, A)$ under strict real-equivalence determines the direct sum asserted in Theorem {\rm\ref{main_th}} uniquely up to permutation of summands, multiplication of any direct summand of type $\Gamma_k \otimes [\pm 1]$ by $-1$, and multiplication of the parameter $a$ in any direct summand of type $\Gamma_k \otimes \, C(a,b)$ by $-1$. For any direct summand $B$ of {\rm Types (i), (ii), (iii), and (iv)} in Theorem {\rm\ref{main_th}}, the {\rm real-KCF}$(B^\top, B)$ under strict
real-equivalence is given in the following table:
\begin{center}
    {\renewcommand{\arraystretch}{2.2}
    \renewcommand{\tabcolsep}{0.1cm}
    \begin{tabular}{|c|l|}\hline
          Block $B$ in {\rm Th. \ref{main_th}}& \phantom{aaaaaaaaaaaaa}  {\rm real-KCF}$(B^\top,B)$\\ \hline \hline
         $J_k(0)$ &
         \begin{tabular}{ll}
          $(F_\ell , G_\ell) \oplus (F_\ell^\top , G_\ell^\top)$   &  if $k=2\ell +1$ \\
          $(J_\ell (0), I_\ell ) \oplus (I_\ell , J_\ell (0))$    &  if $k=2\ell$
         \end{tabular}
         \\ \hline
         $\Gamma_k\otimes (\pm1)$ & $\left(I_k \, , \, J_k \left( (-1)^{k+1} \right) \, \right)$ \\\hline
          $\Gamma_k\otimes C(a,b)$,\
         \begin{tabular}{l}
          $a,b\in\R$, \\[-0.4cm] $a^2+b^2=1$, $b>0$ \end{tabular}  &
         \begin{tabular}{ll}
         $\left(I_k \, , \, J_k \left( (-1)^{k} \right) \, \right) \oplus \left(I_k \, , \, J_k \left( (-1)^{k} \right) \, \right)$     &  if $a=0,b=1$  \\
         $\left(I_{2k} , C_{2k} (\, (-1)^{k+1} (a^2-b^2) \, , \, 2 |ab| \,)\right)$ & if $a\ne 0$
         \end{tabular}
        \\[0.2cm] \hline
         $H_{2k}(a)$,\  $a\in\R$, $0<|a|<1$   & $\displaystyle (I_k \, , \, J_k (a)\,)  \oplus  \left( I_k \, , \, J_k \left( 1/a \right) \, \right)$ \\[0.2cm] \hline
         $\widehat H_{4k}(a,b)$,\  \begin{tabular}{l} $a,b\in\R$,\\[-0.4cm] $a^2+b^2<1$, $b>0$ \end{tabular} & $\displaystyle (I_{2k}, C_{2k} (a,b) ) \oplus \left( I_{2k}, C_{2k} \left(\frac{a}{a^2 + b^2}, \frac{b}{a^2 + b^2} \right) \right)$
         \\[0.2cm] \hline
    \end{tabular}}
\end{center}
\end{theorem}

\begin{remark} \label{rem.recovereal} Observe that, given a matrix pair $(I_{2k}, C_{2k} (c,d))$, $c,d\in \R$, $c^2 + d^2 = 1$, and $d >0$, the equality
$$(I_{2k}, C_{2k} (c,d)) = \left(I_{2k} \, , \, C_{2k} (\, (-1)^{k+1} (a^2-b^2) \, , \, 2 |ab| \,)\right),$$
with $a,b\in \R$, $a^2 + b^2 = 1$, and $b >0$, holds if and only if
$$
a = \pm \sqrt{\frac{1+ (-1)^{k+1} c}{2}} \quad \mbox{and} \quad b = \frac{d}{\sqrt{2 \left(1+ (-1)^{k+1} c \right)}}.
$$
This allows us to determine explicitly, up to the sign of $a$, the block $B = \Gamma_k \otimes C(a,b)$ from the {\rm real-KCF}$(B^\top , B)$.
\end{remark}
\noindent
{\em Proof of Theorem {\rm\ref{thm.1strealCFandKCF}}}. An argument analogous to that at the beginning of the proof of Theorem \ref{hs_th} (b) shows that it suffices to prove the results in the table in the statement. For this purpose, we will perform general strict equivalences, i.e., multiplications by invertible matrices that may be complex, on each real pair $(B^\top,B)$ in the table  until we get the desired real target pair. Then, we apply Lemma \ref{lemm.strictrealpenc} to conclude that $(B^\top,B)$ and the target pair are strictly real-equivalent.

\medskip \noindent (1)
The real-KCF of $(J_k (0)^\top , J_k (0) )$ follows from the first row in the table of Theorem \ref{hs_th} (b) and Lemma \ref{lemm.strictrealpenc}.

\medskip \noindent (2)
The real-KCF of $(\, (\Gamma_k \otimes N)^\top , \Gamma_k \otimes N )$ with $N=\pm 1$ follows from the second row in the table of Theorem \ref{hs_th} (b) with $\mu = \pm 1$ and Lemma \ref{lemm.strictrealpenc}.

\medskip \noindent (3)
Next, we obtain the real-KCF of $(\, (\Gamma_k \otimes N)^\top , \Gamma_k \otimes N )$ with $N= C(a,b)$, with $a,b \in \R$, $a^2 + b^2 =1$, and $b>0$. Observe that in this case $C(a,b)^\top = C(a,b)^{-1}$. Thus,
\begin{align}
(\, (\Gamma_k \otimes C(a,b))^\top , \Gamma_k \otimes C(a,b) ) & = ( \Gamma_k^\top \otimes C(a,b)^\top , \Gamma_k \otimes C(a,b) ) \approx (I_{2k} ,  \Gamma_k^{-\top} \Gamma_k \otimes C(a,b)^2 ) \nonumber \\ & \approx (I_{2k} , C(a,b)^2 \otimes \Gamma_k^{-\top} \Gamma_k), \label{eq.auxx1rconrKCF}
\end{align}
where the last strict equivalence follows from Lemma \ref{lemm.commukron}. If $W$ and $D(\mu)$ are the matrices in Lemma \ref{mu_lemma}, $C(a,b)^2 = W D(\mu)^2 W^*$, which combined with \eqref{eq.auxx1rconrKCF} and Lemma \ref{lemm.problocksHS}, yields
\begin{equation} \label{eq.auxx2rconrKCF}
(\, (\Gamma_k \otimes C(a,b))^\top , \Gamma_k \otimes C(a,b) ) \approx (I_{2k} ,  D(\mu)^2 \otimes J_k ((-1)^{k+1}) ).
\end{equation}
If $a = 0$, then $b=1$, $\mu^2 = (\ii b)^2 = -1$,  $\overline{\mu}^2 = (-\ii b)^2 = -1$, $D(\mu)^2 = -I_2$, and \eqref{eq.auxx2rconrKCF} reads
$$
(\, (\Gamma_k \otimes C(a,b))^\top , \Gamma_k \otimes C(a,b) ) \approx (I_{2k} ,  J_k ((-1)^{k}) \oplus J_k ((-1)^{k})),
$$
which is the desired real-KCF. Then, the use of Lemma \ref{lemm.strictrealpenc} completes the proof. On the other hand, if $a\ne 0$, then \eqref{eq.auxx2rconrKCF} implies
\begin{align*}
(\, (\Gamma_k \otimes C(a,b))^\top , \Gamma_k \otimes C(a,b) ) & \approx (I_{2k} \, ,  \, J_k ((-1)^{k+1} \, \mu^2 )
\oplus J_k ((-1)^{k+1} \, \overline{\mu}^2) ) \\ & \approx (\, I_{2k} \, , \, C_{2k} (\, (-1)^{k+1} (a^2-b^2) \, , \,  (-1)^{k+1} \, (2ab) \,) \, ),
\end{align*}
where the last strict equivalence follows from Lemma \ref{cab_lemma} with $\mu$ replaced by $(-1)^{k+1} \, \mu^2$. If $ (-1)^{k+1} \, (2ab) = 2 |ab|$, we have obtained the target real-KCF. If not, multiply the pair above on the left and on the right by $I_k \otimes \left[ \begin{smallmatrix}
0 & 1 \\ 1 & 0
\end{smallmatrix}\right]$. In one case or in another, the use of Lemma \ref{lemm.strictrealpenc} completes again the proof.

\medskip \noindent (4) The real-KCF of  $(H_{2k} (a)^\top , H_{2k} (a))$, with $a \in \R$, $0<|a| <1$, follows from the last row in the table of Theorem \ref{hs_th} (b) and Lemma \ref{lemm.strictrealpenc} (recall Remark \ref{mu_rem}).

\medskip \noindent (5) Finally, we obtain the real-KCF of $(\widehat{H}_{4k} (a,b)^\top , \widehat{H}_{4k} (a,b))$, with $a,b \in \R$, $a^2 + b^2 <1$, and $b>0$. For this purpose, note that from {\bf Step 2} in the proof of Theorem \ref{main_th} (observe that $\widehat{H}_{4k} (a,b)^\top = \widehat{H}_{4k} (a,b)^*$), Lemma \ref{lemm.problocksHS} and Lemma \ref{cab_lemma}, we have the following
\begin{align*}
\left(\widehat{H}_{4k} (a,b)^\top , \widehat{H}_{4k} (a,b)\right) & \approx
\left( \, \begin{bmatrix}
     H_{2k}(\mu)^* &0\\0&H_{2k}(\overline\mu)^*
\end{bmatrix} \, , \, \begin{bmatrix}
     H_{2k}(\mu)&0\\0&H_{2k}(\overline\mu)
\end{bmatrix} \, \right) \\
 & \approx
 \left(\, I_{4k} \, , \, \begin{bmatrix}
     H_{2k}(\mu)^{-*} H_{2k}(\mu)&0\\0& H_{2k}(\overline\mu)^{-*} H_{2k}(\overline\mu)
\end{bmatrix} \, \right) \\
&  \approx (\, I_{4k} \, , \, J_k (\mu) \oplus J_k (1/\overline\mu) \oplus J_k (\overline\mu) \oplus J_k (1/\mu) \,) \\
&  \approx (\, I_{4k} \, , \, J_k (\mu)  \oplus J_k (\overline\mu) \oplus J_k (1/\overline\mu) \oplus J_k (1/\mu) \,) \\
&  \approx \left(\, I_{4k} \, , \,C_{2k} (a,b)  \oplus C_{2k} \left(\frac{a}{a^2 + b^2}, \frac{b}{a^2 + b^2} \right) \,\right),
\end{align*}
which is the target real-KCF. Again, the use of Lemma \ref{lemm.strictrealpenc} completes the proof.
\qed

\section{Second canonical form of real matrices for real-congruence: block tridiagonal form} \label{sec.secondcan}

Theorem \ref{second_main_th} establishes a real counterpart of the FHSCF presented in Theorem \ref{thm.fhs} (a). It is one of the two main results of this paper. Theorem \ref{second_main_th} will be complemented in Theorem \ref{thm.2strealCFandKCF} with the relation between the second canonical form of $A\in \R^{n\times n}$ for real-congruence and the real-KCF of $(A^\top,A)$. The relation between the real canonical forms in Theorems \ref{main_th} and \ref{second_main_th} is established in Section \ref{sec.comparisontworeals}. Recall that the matrix $\widehat{T}_{4k} (a,b)$ is the matrix defined in \eqref{eq.blocktri3} with $k$ replaced by $2k$.

\begin{theorem}\label{second_main_th}
    Each square matrix $A\in\R^{n\times n}$ is real-congruent to a direct sum, uniquely determined up to permutation of summands, of canonical matrices of the following three types
 \begin{center}
    {\renewcommand{\arraystretch}{1.8}
     \renewcommand{\tabcolsep}{0.3cm}
    \begin{tabular}{|c|c|l|}\hline Name&Block&\phantom{aaaaa}Conditions\\\hline\hline
        {\rm  Type Tri-(i)} & $T_k(a)$  &
         \begin{tabular}{l}
         $a \in \R$, $a \ne \pm 1$, \\[-0.2cm]
         each nonzero $a$ is determined \\[-0.2cm]
         up to replacement by $a^{-1}$,\\[-0.2cm]
         $a = 0$ if $k$ is odd
         \end{tabular} \\[0.2cm] \hline
          {\rm Type Tri-(ii)} & $\widehat{T}_{4k} (a,b)$ &\begin{tabular}{l}
         $a,b\in\R$, $a^2+b^2 \ne 1$, and $b>0$, \\[-0.2cm]
         $(a,b)$ is determined up to \\
         replacement by  $\displaystyle \left(\frac{a}{a^2 + b^2} \, , \, \frac{b}{a^2 + b^2}\right)$
         \\[0.3cm]
         \end{tabular} \\[0.2cm] \hline
        {\rm   Type Tri-(iii)} & $\widetilde{\Gamma}_k \otimes N$ & \begin{tabular}{l} where $N = \pm1$ or \\[-0.2cm]
         $N = C(a,b)$, with $a,b\in\R$, $a^2+b^2=1$, and $b>0$ \end{tabular}  \\\hline
    \end{tabular}}
    \end{center}
In the matrices of {\rm Type Tri-(iii)}, $\widetilde{\Gamma}_k \otimes N$ can be replaced by $N \otimes \widetilde{\Gamma}_k$.
\end{theorem}
\begin{proof}
The proof is very similar to that of Theorem \ref{main_th}, but based on the FHSCF in Theorem \ref{thm.fhs} (a) instead of the HSCF of Theorem \ref{hs_th} (a). Therefore, we focus on the relevant differences and sketch very briefly the similar parts. According to Theorem \ref{thm.fhs} (a) there exists a nonsingular matrix $P\in \C^{n \times n}$ such that $A = P F_A P^*$, where $F_A$ is a direct sum of canonical matrices of types Tri-I and Tri-II. The fact that $A$ is real implies via the same argument in the first paragraph of the proof of Theorem \ref{main_th} that for $\mu \in \C\setminus \R$ the blocks $T_k (\mu)$,  $T_k (\overline\mu)$ are paired up in $F_A$ (note that $k$ is even in such $T_k (\mu)$ because $0 \notin \C\setminus \R$), as well as the blocks $\mu\, \widetilde{\Gamma}_k,\overline\mu\,\widetilde{\Gamma}_k$.

Then, we divide the proof in four steps as in Theorem \ref{main_th}. Steps 3 and 4 are identical to those in Theorem \ref{main_th}, and Step 1 is also equal except for the fact that $\Gamma_k$ is now replaced by $\widetilde{\Gamma}_k$. This step proves that every block $\left[\begin{smallmatrix}
    \mu\, \widetilde\Gamma_k&0\\0&\overline\mu\, \widetilde\Gamma_k
\end{smallmatrix}\right]$, with  $\mu \in \C\setminus \R$,  is $^*$\!congruent to a block $\widetilde{\Gamma}_k \otimes C(a,b)$ with $b>0$ as in the statement. Then, the only difference with respect to the proof of Theorem \ref{main_th} is in Step 2, which now deals with $\left[\begin{smallmatrix}
    T_{k}(\mu)&0\\0&T_{k}(\overline\mu)
\end{smallmatrix}\right]$, where $\mu=a+\ii b$, $a, b \in \R$, $b\ne 0$, $|\mu| \ne 1$, and $k$ even. According to Lemma \ref{lemm.reltridiablocks}, $\left[\begin{smallmatrix}
    T_{k}(\mu)&0\\0&T_{k}(\overline\mu)
\end{smallmatrix}\right]$ is (unitarily) $^*$\!congruent to
$\widehat{T}_{2k} (a,b)$. If $b>0$, the proof is complete. Otherwise, perform the (unitary) $^*$congruence $(I_k \otimes \left[\begin{smallmatrix}
0 & 1 \\ 1 & 0
\end{smallmatrix} \right] ) \, \widehat{T}_{2k} (a,b) \, (I_k \otimes \left[\begin{smallmatrix}
0 & 1 \\ 1 & 0
\end{smallmatrix} \right])^\top = \widehat{T}_{2k} (a,-b)$. Since $k$ is even, we replace $\widehat{T}_{2k} (a,b)$ by $\widehat{T}_{4k} (a,b)$ in the statement.
\end{proof}

\subsection{Relation of the second canonical form for real-congruence of $A$ and
the real-KCF of $(A^\top, A)$}
Theorem \ref{thm.2strealCFandKCF} establishes that real-KCF$(A^\top , A)$ determines the canonical form in Theorem \ref{second_main_th} up to the signs of some parameters in the blocks of Type Tri-(iii).
\begin{theorem} \label{thm.2strealCFandKCF}
Let $A \in \R^{n\times n}$. The direct sum asserted in Theorem {\rm\ref{second_main_th}} determines the {\rm real-KCF}$(A^\top, A)$ under strict real-equivalence uniquely up to permutation of its direct summands. Conversely, the {\rm real-KCF}$(A^\top, A)$ under strict real-equivalence determines the direct sum asserted in Theorem {\rm\ref{second_main_th}} uniquely up to permutation of summands, multiplication of any direct summand of type $\widetilde{\Gamma}_k \otimes [\pm 1]$ by $-1$, and multiplication of the parameter $a$ in any direct summand of type $\widetilde{\Gamma}_k \otimes \, C(a,b)$ by $-1$. For any direct summand $B$ of {\rm Types Tri-(i), Tri-(ii), and Tri-(iii)} in Theorem {\rm\ref{second_main_th}}, the {\rm real-KCF}$(B^\top, B)$ under strict real-equivalence is given in the following table:
\begin{center}
    {\renewcommand{\arraystretch}{2.3}
    \renewcommand{\tabcolsep}{0.1cm}
    \begin{tabular}{|c|l|}\hline
         \begin{tabular}{c} Block $B$ in {\rm Th. \ref{second_main_th}}\end{tabular}& \phantom{aaaaaaaaaaaaaaaaa} {\rm real-KCF}$(B^\top,B)$\\ \hline \hline
          $T_k(a)$\,,\
         \begin{tabular}{l}
         $a \in \R$, $a \ne \pm 1$, \\[-0.3cm]
         $a = 0$ if $k$ is odd
         \end{tabular}&
         \begin{tabular}{ll}
          $(F_\ell , G_\ell) \oplus (F_\ell^\top , G_\ell^\top)$   &  if $k=2\ell +1$ \\[-0.2cm]
          $(J_\ell (0), I_\ell ) \oplus (I_\ell , J_\ell (0))$    &  if $k=2\ell$, $a=0$ \\[-0.1cm]
           $\displaystyle (I_\ell  ,  J_\ell (a) ) \oplus \left(I_\ell \, , \, J_\ell \left( 1/a \right) \right)$    &  if $k=2\ell$, $a\ne 0$ \\[0.2cm]
         \end{tabular}
         \\[0.2cm]  \hline
         $\widehat{T}_{4k} (a,b)$\,,\ \begin{tabular}{l} $a,b\in\R$,  \\[-0.3cm]
         $a^2+b^2 \ne 1$, $b>0$ \\[0.1cm]
         \end{tabular}
         & $\displaystyle (I_{2k}, C_{2k} (a,b) ) \oplus \left( I_{2k}, C_{2k} \left(\frac{a}{a^2 + b^2}, \frac{b}{a^2 + b^2} \right) \right)$
          \\[0.2cm]  \hline
         $\widetilde{\Gamma}_k\otimes [\pm1]$& $\left(I_k \, , \, J_k \left( (-1)^{k+1} \right) \, \right)$ \\\hline
         $\widetilde\Gamma_k\otimes C(a,b)$\,,\
         \begin{tabular}{l}
         $a,b\in\R$, \\[-0.4cm] $a^2+b^2=1$, $b>0$ \end{tabular}  &
         \begin{tabular}{ll}
         $\left((I_k \, , \, J_k \left( (-1)^{k} \right) \, ) \oplus (I_k \, , \, J_k \left( (-1)^{k} \right) \, \right)$     &  if $a=0,b=1$  \\
         $\left(I_{2k} , C_{2k} (\, (-1)^{k+1} (a^2-b^2) \, , \, 2 |ab| \,)\right)$ & if $a\ne 0$
         \end{tabular}
        \\[0.3cm] \hline
    \end{tabular}}
\end{center}
\end{theorem}
\begin{proof} As in the proof of Theorem \ref{thm.1strealCFandKCF}, it suffices to prove the results in the table in the statement. We will proceed case by case via analogous manipulations to those in Theorem \ref{thm.1strealCFandKCF}.

Before we start note that the second row in the table of Theorem \ref{thm.fhs} (b) with $\mu=1$ implies $(\widetilde{\Gamma}_k^\top , \widetilde{\Gamma}_k) \approx (I_k ,J_k ((-1)^{k+1}))$. On the other hand,  $(\widetilde{\Gamma}_k^\top , \widetilde{\Gamma}_k) \approx (I_k , \widetilde{\Gamma}_k^{-\top} \widetilde{\Gamma}_k )$. Thus, $(I_k , \widetilde{\Gamma}_k^{-\top} \widetilde{\Gamma}_k) \approx (I_k ,J_k ((-1)^{k+1}))$, which implies that $\widetilde{\Gamma}_k^{-\top} \widetilde{\Gamma}_k$ is similar to $J_k ((-1)^{k+1})$. We will use this fact in the proof without explicitly referring to.

\medskip \noindent (1) The real-KCF of $(T_k(a)^\top , T_k (a))$ with $a\in \R$, $a\ne \pm 1$ follows from the first row in the table in Theorem \ref{thm.fhs} (b) with $\mu =a$ and Lemma \ref{lemm.strictrealpenc}.

\medskip \noindent (2) We obtain the real-KCF of $(\widehat{T}_{4k} (a,b)^\top , \widehat{T}_{4 k} (a,b))$, $a,b \in \R$, $a^2 + b^2 \ne 1$, $b>0$, from Lemma \ref{lemm.reltridiablocks}, the table in Theorem \ref{thm.fhs} (b) with $\mu = a + \ii b$, and Lemma \ref{cab_lemma} as follows
\begin{align*}
(\widehat{T}_{4k} (a,b)^\top , \widehat{T}_{4 k} (a,b)) & \approx  (T_{2k} (\mu)^* \oplus T_{2k} (\overline \mu)^* , T_{2k} (\mu) \oplus T_{2k} (\overline \mu)) \\ & = (T_{2k} (\mu)^* , T_{2k} (\mu)) \oplus (T_{2k} (\overline \mu)^* , T_{2k} (\overline \mu)) \\
 & \approx  (I_k \, , \, J_k (\mu)\,)  \oplus  \left(I_k \, , \, J_k \left( 1/\overline{\mu} \right) \, \right) \oplus (I_k \, , \, J_k (\overline \mu)\,)  \oplus  \left(I_k \, , \, J_k \left( 1/\mu \right) \, \right) \\
 & \approx  (I_k \, , \, J_k (\mu)\,)  \oplus (I_k \, , \, J_k (\overline \mu)\,)  \oplus  \left(I_k \, , \, J_k \left( 1/\overline{\mu} \right) \, \right) \oplus  \left(I_k \, , \, J_k \left( 1/\mu \right) \, \right)  \\
 & \approx (I_{2k}, C_{2k} (a,b) ) \oplus \left( I_{2k}, C_{2k} \left(a/(a^2 + b^2), b/(a^2 + b^2) \right) \right).
\end{align*}
The proof is completed by applying Lemma \ref{lemm.strictrealpenc}.

\medskip \noindent (3)  The real-KCF of $(\, (\widetilde{\Gamma}_k \otimes N)^\top , \widetilde{\Gamma}_k \otimes N )$ with $N=\pm 1$ follows from the second row in the table of Theorem \ref{thm.fhs} (b) with $\mu = \pm 1$ and Lemma \ref{lemm.strictrealpenc}.

\medskip \noindent (4) Next, we obtain the real-KCF of $(\, (\widetilde\Gamma_k \otimes C(a,b))^\top , \widetilde\Gamma_k \otimes C(a,b) )$, with $a,b \in \R$, $a^2 + b^2 =1$, and $b>0$. We proceed as in the proof of Theorem \ref{thm.1strealCFandKCF} for $\Gamma_k \otimes C(a,b)$, i.e., case (3) in that proof. So, we skip most of the details. Recall that for these parameters $a,b$, $C(a,b)^\top = C(a,b)^{-1}$. Thus,
\begin{align*}
(\, (\widetilde\Gamma_k \otimes C(a,b))^\top , \widetilde\Gamma_k \otimes C(a,b) ) & \approx (I_{2k} ,  \widetilde\Gamma_k^{-\top} \widetilde\Gamma_k \otimes C(a,b)^2 ) \approx (I_{2k} , C(a,b)^2 \otimes \widetilde\Gamma_k^{-\top} \widetilde\Gamma_k) \\
& \approx (I_{2k} ,  D(\mu)^2 \otimes J_k ((-1)^{k+1}) ),
\end{align*}
where $D(\mu)$ is the matrix in Lemma \ref{mu_lemma}. Note that this strict equivalence is as that in \eqref{eq.auxx2rconrKCF}. Therefore, the rest of the proof is exactly the same as the corresponding one in Theorem \ref{thm.1strealCFandKCF}.
\end{proof}

\section{Relation between the first and the second canonical forms for real congruence} \label{sec.comparisontworeals}

Theorem \ref{thm.relation2} presents the precise real-congruence relations between the canonical blocks in Theorems \ref{main_th} and \ref{second_main_th}. It is a direct consequence of Lemma \ref{equivalencegammas.th} and Theorems \ref{thm.1strealCFandKCF} and \ref{thm.2strealCFandKCF}, so the proof is omitted.

\begin{theorem} \label{thm.relation2} Let $J_k (0), \Gamma_k, H_{2k} (a),  C(a,b), \widehat{H}_{4k} (a,b), T_k(a), \widetilde{\Gamma}_k$, and $\widehat{T}_{2k} (a,b)$ be the matrices in \eqref{jordanblock}-\eqref{eq.blocktri3}, and
\begin{equation*}
        \varepsilon :=\left\{\begin{array}{rc}
             1&\mbox{if $k\equiv1,2$ {\rm(mod 4)}},  \\
             -1&\mbox{if $k\equiv0,3$ {\rm (mod 4)}.}
        \end{array}\right.
\end{equation*}
Then the real-congruences described in the following table hold.
\begin{center}
    {\renewcommand{\arraystretch}{2.2}
    \renewcommand{\tabcolsep}{0.11cm}
    \begin{tabular}{|c|c|}\hline
         \phantom{aaa} Block $B$ in {\rm Theorem \ref{main_th}} & is real-congruent to block $C$ in {\rm Theorem \ref{second_main_th}} \\ \hline \hline
         $J_k(0)$ &
         $T_k (0)$
         \\ \hline
         $\Gamma_k\otimes [\pm1]$& $\widetilde{\Gamma}_k \otimes [\pm\varepsilon]$ \\ \hline
         $\Gamma_k\otimes C(a,b)$,\ \begin{tabular}{l} $a,b\in\R$, \\[-0.4cm] $a^2+b^2=1$, $b>0$ \end{tabular}  &
            $\widetilde{\Gamma}_k \otimes C(\varepsilon a, b)$
        \\[0.2cm] \hline
         $H_{2k}(a)$,  $a\in\R$, $0<|a|<1$   & $T_{2k} (a)$ \\[0.2cm] \hline
         $\widehat H_{4k}(a,b)$,  \begin{tabular}{l} $a,b\in\R$,\\[-0.2cm] $a^2+b^2<1$, $b>0$ \end{tabular} & $\widehat{T}_{4k} (a,b)$
         \\[0.2cm] \hline
    \end{tabular}}
\end{center}
\end{theorem}

\section{Correspondence with the blocks in the Lee-Weinberg canonical form}\label{leewein_sec}
In this section, we relate the blocks in the first canonical form for real-congruence in Theorem \ref{main_th} with those in the real-congruence canonical form in \cite[Theorem II, p. 213]{lw96}.  Once this result is established, the correspondence of the blocks in \cite[Theorem II, p. 213]{lw96} with those in the second block-tridiagonal canonical form for real-congruence in Theorem \ref{second_main_th} follows immediately from Theorem \ref{thm.relation2}. For brevity, such correspondence is not explicitly stated.

To describe the canonical form in \cite{lw96}, several structured matrices must be defined. We use the same notation as in \cite{lw96}, although additional restrictions are imposed on some parameters to make the form in \cite{lw96} truly canonical (see Remark \ref{rem.leeweinparam}). We start with the following auxiliary matrices:
\begin{align*}
 &   L_k:=\begin{bmatrix}
        1\\1&1\\&\ddots&\ddots\\&&1&1\\&&&1
    \end{bmatrix}_{(k+1)\times k},\quad L_k^+:=\begin{bmatrix}
        1&-1\\&1&-1\\&&\ddots&\ddots\\&&&1&-1
    \end{bmatrix}_{k\times (k+1)},\\
 &   \Delta_k:=\begin{bmatrix}
        &&1\\&\iddots&\\1&&
    \end{bmatrix}_{k\times k},\quad \Lambda_k:=\begin{bmatrix}
        &&&0\\&&0&1\\&\iddots&1\\0&1&&
    \end{bmatrix}_{k\times k}  \qquad (\Lambda_1 := 0),\\
  &  S\Delta_k:=\left\{\begin{array}{lc}\begin{bmatrix}
        0&\Delta_{k/2}\\-\Delta_{k/2}&0
    \end{bmatrix}_{k\times k} &\mbox{($k$ even)},\\[0.5cm]  \begin{bmatrix}
        0&0&0\\0&0&\Delta_{(k-1)/2}\\0&-\Delta_{(k-1)/2} & 0
    \end{bmatrix}_{k\times k}&\mbox{($k$ odd)}.\end{array}\right.
\end{align*}

The canonical form for real-congruence of real square matrices in \cite{lw96} is a direct sum of the following eight types of blocks:
\begin{align}
   m_3' &:=\begin{bmatrix}
       0&L_k\\L_k^+&0
   \end{bmatrix}_{(2k+1)\times(2k+1)} \qquad (m_3' := 0 \;\; \mbox{if} \;\; k=0),\label{m3}\\
   \infty_4' & :=\varepsilon \, (S\Delta_k+\Lambda_k)\qquad \mbox{($\varepsilon = \pm 1$, $k$ even)},\label{infty4}\\
   \infty_5' & :=\begin{bmatrix}
       0&\Delta_k+\Lambda_k\\-\Delta_k+\Lambda_k&0
   \end{bmatrix}_{2k \times 2k} \qquad \mbox{($k$ odd)},\label{infty5}\\
   o_3' & :=\varepsilon \, (\Delta_k+S\Delta_k)\qquad \mbox{($\varepsilon = \pm 1$, $k$ odd)},\label{o3}\\
   o_4'& :=\begin{bmatrix}
       0&\Delta_k+\Lambda_k\\\Delta_k-\Lambda_k&0
   \end{bmatrix}_{2k \times 2k} \qquad \mbox{($k$ even)},\label{o4}\\
   \alpha_3' & :=\begin{bmatrix}
       0&(\alpha+1)\Delta_k+\Lambda_k\\
       (-\alpha+1)\Delta_k-\Lambda_k&0
   \end{bmatrix}_{2k \times 2k} \qquad (\alpha \in \R, \; \alpha > 0),\label{alpha3}
\end{align}
\begin{align}
   \beta_4'& :=\varepsilon \, \begin{bmatrix}
       &&&R\\&&R&S\\&\iddots&\iddots\\R&S
   \end{bmatrix}_{2k\times2k} \qquad (\varepsilon = \pm 1),\label{beta4}\\
   \beta_5' & :=\left[\begin{array}{cccc|cccc}
        &&&&&&&R'\\&&&&&&R'&S'\\&&&&&\iddots&\iddots\\&&&&R'&S'\\\hline
        &&&-T&&&&\\&&-T&-S'&&&&\\&\iddots&\iddots&&&&\\-T&-S'&&&&&&
   \end{array}\right]_{4k\times4k},\label{beta5}
\end{align}
where
$$
R=\begin{bmatrix}
    1&|b|\\-|b|&1
\end{bmatrix},\quad S=\begin{bmatrix}
    0&1\\-1&0
\end{bmatrix},\quad T=\begin{bmatrix}
   b&a-1\\a-1&-b
\end{bmatrix},\quad R'=\begin{bmatrix}
    b&a+1\\a+1&-b
\end{bmatrix},\quad S'=\begin{bmatrix}
    0&1\\1&0
\end{bmatrix}
$$
with
$$
a,b \in \R \quad \mbox{and} \quad a \ne 0, \, b \ne 0.
$$

Theorem \ref{thm.lw96} is the main result in \cite{lw96}.

\begin{theorem} \label{thm.lw96} {\rm \cite[Theorem II, p. 213]{lw96}}
Each square real matrix is real-congruent to a direct sum, uniquely determined up to permutation of summands, of matrices of types $m_3', \infty_4', \infty_5', o_3', o_4',$ $\alpha_3', \beta_4'$, and $\beta_5'$.
\end{theorem}

\begin{remark} {\rm (On the values of the parameters $\alpha$ in \eqref{alpha3} and $a, b$ in \eqref{beta4}--\eqref{beta5})} \label{rem.leeweinparam}
In {\rm\cite{lw96}} the parameter $\alpha$ in \eqref{alpha3} is just required to be ``finite and nonzero''. We have imposed $\alpha >0$, because otherwise the block would not be ``truly'' canonical. The reason is that $\alpha_3'$ is real-congruent to the matrix obtained by replacing $\alpha$ by $-\alpha$ in $\alpha_3'$. To see this, set
$S_k := \diag (1, -1, 1, -1, \ldots, (-1)^{k-1} ) \in \R^{k \times k}$ and $\widetilde S_k := \diag ((-1)^{k-1} , \ldots, -1, 1, -1, 1) \in \R^{k \times k}$, i.e., $\widetilde S_k$ has the same diagonal entries as $S_k$ but in reversed order. Then
$$
\begin{bmatrix}
0 & I_k \\
I_k & 0
\end{bmatrix}
\begin{bmatrix}
S_k & 0\\
0 & \widetilde S_k
\end{bmatrix}
\, \alpha_3'
\begin{bmatrix}
S_k & 0\\
0 & \widetilde S_k
\end{bmatrix}
\begin{bmatrix}
0 & I_k \\
I_k & 0
\end{bmatrix}
= \begin{bmatrix}
       0&(-\alpha+1)\Delta_k+\Lambda_k\\
       (\alpha+1)\Delta_k-\Lambda_k&0
   \end{bmatrix}.
$$

In {\rm\cite{lw96}}, the parameters $a$ and $b$ in $\beta_4'$, and $\beta_5'$ are just required to be real. However, {\rm\cite[Theorem II, p. 213]{lw96}} is obtained from {\rm\cite[Theorem 2 (c), p. 344]{thompson}} and this result requires $a\ne 0\ne b$. Therefore, we also have required $a\ne 0\ne b$. In fact, if $a=0$ or $b=0$, then $\beta_4'$ and $\beta_5'$ would not be canonical blocks since they would be real-congruent to direct sums of other canonical blocks in Theorem {\rm\ref{thm.lw96}}. To see this, compare the table in Appendix {\rm\ref{append.degenerate}} with Table {\rm\ref{correspondence_table}}.
\end{remark}

\medskip

Theorem \ref{thm.correspondlw96} is the main theorem of this section.

\begin{theorem} \label{thm.correspondlw96}
Let $m_3', \infty_4', \infty_5', o_3', o_4',$ $\alpha_3', \beta_4',$ and $\beta_5'$ be the matrices in \eqref{m3}--\eqref{beta5}, and $J_k (0)$, $\Gamma_k$, $H_{2k} (a), C(a,b)$, and $\widehat{H}_{4k} (a,b)$ be the matrices defined in \eqref{jordanblock}--\eqref{h4k} and appearing in Theorem {\rm\ref{main_th}}. Then the real congruences described in Table {\rm\ref{correspondence_table}} hold.
\end{theorem}

The proof of Theorem \ref{thm.correspondlw96} is postponed to Subsection \ref{sec.proofofcorresp}, where it is obtained as a corollary of a sequence of lemmas, each proving the real-congruence described in one of the rows of Table \ref{correspondence_table}. We hasten to admit that the precise sign of the parameters in two of the blocks in the right column of Table \ref{correspondence_table} has not been determined. These are the blocks corresponding to $\beta_4'$.

\begin{table}[]
    \centering
    {\renewcommand{\arraystretch}{2.4}
    \renewcommand{\tabcolsep}{0.25cm}
    \begin{tabular}{|c|c|}\hline
         Block $B$ in Theorem \ref{thm.lw96} &is real-congruent to block $C$ in Theorem \ref{main_th}  \\\hline\hline
         $m_3'$&$J_{2k+1}(0)$\\[0.1cm] \hline
         $\infty_4'$& $\Gamma_k \otimes (\varepsilon \, (-1)^{\frac{k}{2} + 1}
         )$ \, ($k$ even) \\[0.1cm] \hline
         $\infty_5'$&$\Gamma_k\otimes C(0,1)$ \, ($k$ odd)\\[0.1cm] \hline
         $o_3'$&$\Gamma_k \otimes  (\varepsilon \, (-1)^{\frac{k-1}{2}}
         )$ \, ($k$ odd) \\[0.1cm] \hline
         $o_4'$&$\Gamma_k\otimes C(0,1)$ \, ($k$ even)\\[0.1cm] \hline
         $\alpha_3'$&\begin{tabular}{ll}
              $J_{2k}(0)$  & if $\alpha = 1$  \\[0.2cm]
              $\displaystyle H_{2k}\left(\frac{1-\alpha}{1+\alpha}\right)$ & if $\alpha \neq 1$ \\[0.3cm]
              \end{tabular}
              \\\hline
         $\beta_4'$&\begin{tabular}{ll} $\displaystyle \Gamma_k\otimes C \left(\pm \frac{|b|}{\sqrt{1+b^2}} \, ,\, \frac{1}{\sqrt{1+b^2}} \right)$ & if $k$ is even
         \\[0.3cm]
         $\displaystyle \Gamma_k\otimes C \left(\pm \frac{1}{\sqrt{1+b^2}}  \, ,\, \frac{|b|}{\sqrt{1+b^2}}  \right)$
         &if $k$ is odd \\[0.3cm] \end{tabular}\\[0.2cm] \hline
         $\beta_5'$&
              $\displaystyle \widehat H_{4k}\left(\frac{1-(a^2+b^2)}{(1+|a|)^2+b^2}\, ,\, \frac{2|b|}{(1+|a|)^2+b^2}\right)$
         \\[0.3cm] \hline
    \end{tabular}}
    \caption{Real-congruences between the canonical blocks in Theorem \ref{thm.lw96} and those in Theorem \ref{main_th}. The sizes of the blocks  in Theorem \ref{thm.lw96} are those indicated in \eqref{m3}--\eqref{beta5}. The parameter $\alpha$ in $\alpha_3'$ satisfies $\alpha
    >0$ and $a, b$ in $\beta_4'$ and $\beta_5'$ satisfy $a\ne 0\ne b$.}
    \label{correspondence_table}
\end{table}

We see in Table \ref{correspondence_table} that the difference in the number of distinct canonical blocks in the canonical forms of Theorems \ref{main_th} and \ref{thm.lw96} comes from the blocks $\infty_4', \infty_5',o_3', o_4',$ and $\beta_4'$, which are gathered in the two variants $\Gamma_k \otimes [\pm 1]$ and $\Gamma_k \otimes C(a,b)$ of the blocks of Type (ii) in Theorem \ref{main_th}.

\subsection{Proof of Theorem \ref{thm.correspondlw96}} \label{sec.proofofcorresp}
In addition to performing direct real-congruences, we will often use the following approach for proving the results in this section: given a block $B$ in the left column of Table \ref{correspondence_table}, we will compute the real-KCF$(B^\top,B)$ and we will make use of the table in Theorem \ref{thm.1strealCFandKCF} to determine the block in Theorem \ref{main_th} to which $B$ is real-congruent. Recall that Lemma \ref{lemm.strictrealpenc} allows us to compute real-KCFs by means of intermediate complex strict equivalences.

\begin{lemma} The matrix $m_3'$ in Table {\rm\ref{correspondence_table}} is real-congruent to $J_{2k+1}(0)$.
\end{lemma}
\begin{proof}
To see this, we compute the real-KCF of the pair $((m_3')^\top , m_3')$:
\begin{equation*}
((m_3')^\top , m_3')  = \left( \begin{bmatrix}
 0 & (L_k^+)^\top \\
 L_k^\top & 0
\end{bmatrix} \, , \,
\begin{bmatrix}
 0 & L_k \\
 L_k^+ & 0
\end{bmatrix}
\right)  \stackrel{r}{\approx} ((L_k^+)^\top, L_k) \oplus (L_k^\top , L_k^+) .
\end{equation*}
Thus, real-KCF$((m_3')^\top , m_3') = \mbox{real-KCF}((L_k^+)^\top \, , \, L_k) \oplus \mbox{real-KCF}(L_k^\top \,  , \, L_k^+)$. We first compute $\mbox{real-KCF}(L_k^\top \,  , \, L_k^+)$. For this purpose, we consider the pencil in the variable $\lambda$
$$
\lambda L_k^\top + L_k^+    =\left[\begin{array}{cccc}
         1+\la&\la-1&&\\
         &\ddots&\ddots&\\
         &&1+\la&\la-1
    \end{array}\right]_{(k+1)\times k}.
    $$
    This pencil has normal rank $k$ and has no finite or infinite eigenvalues, since the rank is $k$ when $\lambda$ is replaced by any number and the rank of $L_k^\top$ is also $k$. Moreover, it has no left minimal indices and has only one right minimal index. By the Index Sum Theorem \cite[Theorem 6.5]{indexsum}, this right minimal index must be $k$. Therefore, $\mbox{real-KCF}(L_k^\top \,  , \, L_k^+) = (F_k,G_k)$. A similar argument proves that $\mbox{real-KCF}((L_k^+)^\top \, , \, L_k) = (F_k^\top,G_k^\top)$. Therefore, real-KCF$((m_3')^\top , m_3') = (F_k,G_k) \oplus (F_k^\top,G_k^\top)$ and the result follows from the first row in the table of Theorem \ref{thm.1strealCFandKCF}.
\end{proof}

\begin{lemma} The matrix $\infty_4'$ in Table {\rm\ref{correspondence_table}} is real-congruent to $\Gamma_k \otimes \left( \varepsilon \, (-1)^{\frac{k}{2} + 1} \right)$.
\end{lemma}
\begin{proof}
Taking into account that $k$ is even note that
 $$
 \infty_4'=\varepsilon \left[\begin{array}{cccc|cccc}
      &&&&&&&1 \\&&&&&&1&1\\&&&&&\iddots&\iddots&\\&&&&1&1&&\\\hline
      &&&-1&1&&&\\&&-1&1&&&&\\&\iddots&\iddots&&&&&\\
      -1&1&&&&&&
 \end{array}\right]_{k\times k}.
 $$
 Now, if we set $S:=\diag((-1)^{k/2},(-1)^{k/2+1},(-1)^{k/2},(-1)^{k/2+1},\hdots)_{k/2\times k/2}\oplus I_{k/2}$, then $S\infty_4' S^\top =\varepsilon \, (-1)^{\frac{k}{2} + 1} \Gamma_k$ (namely, we just change the sign of the first $k/2$ rows and columns with odd or even indices, depending on the parity of $k/2$).
 \end{proof}
\begin{lemma}
The matrix $\infty_5'$ in Table {\rm\ref{correspondence_table}} is real-congruent to $\Gamma_k\otimes C(0,1)$.
\end{lemma}
\begin{proof}
Note that
 $$
 \infty_5'=\left[\begin{array}{cccc|cccc}
      &&&&&&&1 \\&&&&&&1&1\\&&&&&\iddots&\iddots&\\&&&&1&1&&\\\hline
      &&&-1&&&&\\&&-1&1&&&&\\&\iddots&\iddots&&&&&\\
      -1&1&&&&&&
 \end{array}\right]_{2k\times 2k}.
 $$
Taking into account that $k$ is odd define $S:=\diag(1,-1,1,-1, \hdots, 1,-1,1)_{k\times k}\oplus I_k$. Then
$$
S\infty_5'S^\top =\left[\begin{array}{cccccc|ccccc}
      &&&&&&&&&&1 \\&&&&&&&&&-1&-1\\&&&&&&&&\iddots&\iddots&\\&&&&&&&-1&-1\\&&&&&&1&1&&\\\hline
      &&&&&-1&&&&\\&&&&1&1&&&&\\&&&-1&-1&&&&&\\&&\iddots&\iddots&&&&&\\
      &1&1&&&&&\\-1&-1&&&&&&&&&
 \end{array}\right]_{2k\times 2k} = C(0,1) \otimes \Gamma_k ,
$$
and the result follows from Lemma \ref{lemm.commukron}.
\end{proof}

\begin{lemma}
The matrix $o_3'$ in Table {\rm\ref{correspondence_table}} is real-congruent to $\Gamma_k \otimes \left( \varepsilon \, (-1)^{\frac{k-1}{2}}\right)$.
\end{lemma}
\begin{proof}
Note that
 $$
 o_3'=\varepsilon \left[\begin{array}{ccccc|ccccc}
      & & & & & & & &  1 \\&&&&&&&1&1\\&&&&&&\iddots&\iddots\\&&&&&1&1&&\\&&&&1&1\\\hline
      &&&1&-1&&&\\&&1&-1&&&&\\&\iddots&\iddots&&&&&\\
      1&-1&&&&&&
 \end{array}\right]_{k\times k}.
 $$
 Taking into account that $k$ is odd, if $$S:=\diag((-1)^{(k-1)/2},(-1)^{(k+1)/2},(-1)^{(k-1)/2},(-1)^{(k+1)/2},\hdots)_{\frac{k-1}{2}\times\frac{k-1}{2}}\oplus I_{\frac{k+1}{2}},$$ then $So_3'S^\top=\varepsilon \, (-1)^{\frac{k-1}{2}}  \Gamma_k$.
\end{proof}
\begin{lemma}
The matrix $o_4'$ in Table {\rm\ref{correspondence_table}} is real-congruent to $\Gamma_k\otimes C(0,1)$.
\end{lemma}
\begin{proof}
 Note that
 $$
 o_4'=\left[\begin{array}{cccc|cccc}
      &&&&&&&1 \\&&&&&&1&1\\&&&&&\iddots&\iddots&\\&&&&1&1&&\\\hline
      &&&1&&&&\\&&1&-1&&&&\\&\iddots&\iddots&&&&&\\
      1&-1&&&&&&
 \end{array}\right]_{2k\times 2k}.
 $$
Proceeding as for $\infty_5'$ but with $k$ even, if $S:=\diag(-1,1,\hdots,-1,1)_{k\times k}\oplus I_k$, then
$$
So_4'S^\top =\left[\begin{array}{cccccc|cccccc}
      &&&&&&&&&&&-1 \\&&&&&&&&&&1&1\\&&&&&&&&&\iddots&\iddots&\\&&&&&&&&-1&-1\\&&&&&&&1&1&&\\\hline
      &&&&&1&&\\&&&&-1&-1&&&&\\&&&1&1&&&&&\\&&\iddots&\iddots&&&&&\\
      &1&1&&&&&&\\-1&-1&&&&&&&&&&
 \end{array}\right]_{2k\times 2k} = C(0,1) \otimes \Gamma_k ,
$$
and the result follows from Lemma \ref{lemm.commukron}.
\end{proof}

\begin{lemma}
The matrix $\alpha_3'$ in Table {\rm\ref{correspondence_table}} is real-congruent to $J_{2k} (0)$ if $\alpha = 1$ and to $H_{2k}\left(\frac{1-\alpha}{1+\alpha}\right)$ if $\alpha \ne 1$.
\end{lemma}
\begin{proof} We prove this result by computing the real-KCF$( (\alpha_3')^\top , \alpha_3' )$. To this purpose, note first that $\Delta_k ( (\alpha + 1) \Delta_k + \Lambda_k) = J_k (\alpha + 1)$ and
$\Delta_k ( (-\alpha + 1) \Delta_k - \Lambda_k) = -J_k (\alpha - 1)$. Taking into account that $ J_k (\alpha + 1)$ is invertible (because $\alpha >0$), we proceed as follows:
\begin{align*}
( (\alpha_3')^\top , \alpha_3' ) & =
\left(
\begin{bmatrix}
0 & (-\alpha + 1) \Delta_k - \Lambda_k \\
(\alpha + 1) \Delta_k + \Lambda_k & 0
\end{bmatrix} \, , \,
\begin{bmatrix}
0 & (\alpha + 1) \Delta_k + \Lambda_k \\
(-\alpha + 1) \Delta_k - \Lambda_k & 0
\end{bmatrix}
\right) \\
& \approx
\left(
\begin{bmatrix}
0 & -J_k(\alpha - 1)  \\
J_k(\alpha + 1) & 0
\end{bmatrix} \, , \,
\begin{bmatrix}
0 & J_k(\alpha + 1) \\
- J_k(\alpha - 1) & 0
\end{bmatrix}
\right) \\
& \approx
\left(
\begin{bmatrix}
-J_k(\alpha + 1)^{-1}J_k(\alpha - 1) & 0  \\
0 & I_k
\end{bmatrix} \, , \,
\begin{bmatrix}
I_k & 0 \\
0 & - J_k(\alpha + 1)^{-1} J_k(\alpha - 1)
\end{bmatrix}
\right).
\end{align*}
The Jordan canonical form of the matrix $C = - J_k(\alpha + 1)^{-1} J_k(\alpha - 1)$ is $J_k \left(\frac{1-\alpha}{1 + \alpha}\right)$, because $C$ is upper triangular with all diagonal entries equal to $\frac{1-\alpha}{1 + \alpha}$. Hence, $C$ has only one eigenvalue equal to $\frac{1-\alpha}{1 + \alpha}$ with algebraic multiplicity $k$. Moreover, its geometric multiplicity is $1$ because
\begin{align*}
\rank \left(C - \frac{1-\alpha}{1 + \alpha} I_k\right) & = \rank \left(J_k (\alpha - 1) + \frac{1-\alpha}{1 + \alpha} J_k (\alpha+1)\right) \\
& = \rank \left((1+\alpha) J_k (\alpha - 1) + (1-\alpha) J_k (\alpha +1)\right) \\
& = \rank \begin{bmatrix}
0 & 2 & & \\ &\ddots & \ddots & \\&& \ddots & 2 \\ &&& 0
\end{bmatrix} = k-1.
\end{align*}
Thus,
$$
( (\alpha_3')^\top , \alpha_3' ) = \left( J_k \left(\frac{1-\alpha}{1 + \alpha}\right) \,  , \, I_k \right) \oplus
\left( I_k \, , \, J_k \left(\frac{1-\alpha}{1 + \alpha}\right) \right).
$$
The result follows from the table in Theorem \ref{thm.1strealCFandKCF}. Namely, from its first row for $\alpha =1$ and from its third row for $\alpha \ne 1$.
\end{proof}

We will need Lemma \ref{lemm.auxbetas} in the proofs of Lemmas \ref{lemm.beta4} and \ref{lemm.beta5}. In the statement of Lemma \ref{lemm.auxbetas}, $\mathcal{N} (A)$ and $\mathcal{C} (A)$ denote the null and column space of a matrix $A$, respectively.

\begin{lemma} \label{lemm.auxbetas} Let $A, B \in \C^{n\times n}$, with $B$ invertible, $k\geq 2$ be an integer, and
$$
Z_k (A,B) :=
\begin{bmatrix}
&&&& A \\
&&& \iddots &B \\
&&\iddots & \iddots & \\
&A&B& & \\
A & B& &&
\end{bmatrix} \in \C^{nk \times nk}.
$$
Then:
\begin{enumerate}
    \item[\rm 1.] $\dim \mathcal{N} \, (Z_k (A,B)) = \dim \mathcal{N} \,(A)$ if and only if $$\mathcal{N} \, (Z_k (A,B)) = \{ [x_1^\top, 0 , \ldots , 0]^\top \in \C^{nk} \, : \, x_1 \in   \mathcal{N} \,(A)\}.$$
     \item[\rm 2.] $\dim \mathcal{N} \, (Z_k (A,B)) = \dim \mathcal{N} \,(A)$ if and only if $\mathcal{N} \,(A) \cap \mathcal{C} \,(B^{-1} A) = \{0\}$.
\end{enumerate}
\end{lemma}
\begin{proof} It is clear that the following inclusion and equalities hold
\begin{equation} \label{eq.auxnulscols}
\{ [x_1^\top, 0 , \ldots , 0]^\top \in \C^{nk} \, : \, x_1 \in   \mathcal{N} \,(A)\} \subseteq  \mathcal{N} \, (Z_k (A,B)),
\end{equation}
\begin{equation} \label{eq.auxnulscols1}
\dim \{ [x_1^\top, 0 , \ldots , 0]^\top \in \C^{nk} \, : \, x_1 \in   \mathcal{N} \,(A)\} = \dim \mathcal{N} (A),
\end{equation}
\begin{equation} \label{eq.auxnulscols2}
\mathcal{N} \, (Z_k (A,B)) = \left\{
\begin{bmatrix}
x_1 \\ x_2 \\ \vdots \\ x_{k-1} \\ x_k
\end{bmatrix} \, : \, x_i \in \C^n \, \mbox{and} \,
\begin{array}{ll}
     A x_k & =0 ,  \\
     x_k & = - B^{-1} A \, x_{k-1} , \\
     \phantom{Ax} \vdots & \phantom{=} \vdots \\
     x_3 & = - B^{-1} A \, x_2  , \\
     x_2 & =  - B^{-1} A \, x_1 .
\end{array}
\right\} .
\end{equation}

\medskip
\noindent Proof of part 1. If $\dim \mathcal{N} \, (Z_k (A,B)) = \dim \mathcal{N} \,(A)$, then \eqref{eq.auxnulscols} and \eqref{eq.auxnulscols1} together imply that $\{ [x_1^\top, 0 , \ldots , 0]^\top \in \C^{nk} \, : \, x_1 \in   \mathcal{N} \,(A)\} = \mathcal{N} \, (Z_k (A,B))$. The converse follows from \eqref{eq.auxnulscols1}.

\medskip
\noindent Proof of part 2. Assume first that $\dim \mathcal{N} \, (Z_k (A,B)) = \dim \mathcal{N} \,(A)$. If $y \in \mathcal{N} \,(A) \cap \mathcal{C} \,(B^{-1} A)$, then $Ay=0$ and $y = - B^{-1} A \, x$ for some $x\in \C^n$, which implies $[x^\top , y^\top, 0, \ldots , 0]^\top \in \mathcal{N} \, (Z_k (A,B))$ by \eqref{eq.auxnulscols2}. This, in turn, implies $y=0$, by part 1, and $\mathcal{N} \,(A) \cap \mathcal{C} \,(B^{-1} A) = \{ 0 \}$.

Conversely, assume $\mathcal{N} \,(A) \cap \mathcal{C} \,(B^{-1} A) = \{ 0 \}$. If $[x_1^\top , x_2^\top , \ldots , x_k^\top]^\top \in \mathcal{N} \, (Z_k (A,B))$, where $x_i \in \C^n$, then \eqref{eq.auxnulscols2} implies $x_k \in \mathcal{N} \,(A) \cap \mathcal{C} \,(B^{-1} A)$. So $x_k=0$ and again from \eqref{eq.auxnulscols2}, $x_{k-1} \in \mathcal{N} \,(A) \cap \mathcal{C} \,(B^{-1} A)$, which yields $x_{k-1} =0$. The remaining equalities in \eqref{eq.auxnulscols2} give $x_2 = x_3 = \cdots = x_k = 0$ and $Ax_1 =0$, i.e., $\mathcal{N} \, (Z_k (A,B)) = \{ [x_1^\top, 0 , \ldots , 0]^\top \in \C^{nk} \, : \, x_1 \in   \mathcal{N} \,(A)\}.$ The result follows from part 1.
\end{proof}

\begin{lemma} \label{lemm.beta4}
The matrix $\beta_4'$ in Table {\rm\ref{correspondence_table}} is real-congruent to
\begin{align*} \displaystyle \Gamma_k\otimes C \left(\pm \frac{|b|}{\sqrt{1+b^2}} \, ,\, \frac{1}{\sqrt{1+b^2}} \right) & \qquad \mbox{if $k$ is even,}
         \\[0.2cm]
\displaystyle \Gamma_k\otimes C \left(\pm \frac{1}{\sqrt{1+b^2}}  \, ,\, \frac{|b|}{\sqrt{1+b^2}}  \right) &  \qquad \mbox{if $k$ is odd.}
\end{align*}
\end{lemma}
\begin{proof}  Since $\beta_4'$ in \eqref{beta4} only depends on $|b|$, we assume throughout the proof that $b =|b| >0$. The proof proceeds by computing the real-KCF of the pair $( \left( \beta_4' \right)^\top ,  \beta_4')$ and applying Theorem \ref{thm.1strealCFandKCF}. For this purpose, we consider the pencil
\begin{equation} \label{eq.pencilbeta4}
\lambda \left( \beta_4' \right)^\top +  \beta_4' = \varepsilon
\begin{bmatrix}
       &&&\la R^\top +R\\&&\la R^\top + R& (\la -1) S^\top\\&\iddots&\iddots\\\la R^\top + R & (\la -1) S^\top
   \end{bmatrix}_{2k\times2k}
\end{equation}
in the variable $\lambda$. Since $R$ is invertible for any $b>0$, the pencil $\la \left( \beta_4' \right)^\top +  \beta_4'$ has no minimal indices nor eigenvalues at infinity. So, its KCF is determined by the Jordan blocks associated to its finite eigenvalues. The finite eigenvalues of $\la \left( \beta_4' \right)^\top +  \beta_4'$ are the two roots of the polynomial
$$
\det (\la R^\top + R) = (\la +1)^2+b^2(\la-1)^2=(\la + 1 +\ii b(\la -1))(\la + 1-\ii b(\la - 1))
$$
each with algebraic multiplicity $k$. These roots are
\[
\lambda_1 = \frac{b+\ii }{b-\ii} = \frac{b^2 - 1 + \ii \, 2 b}{1+b^2} \quad \mbox{and} \quad \lambda_2 = \overline{\la_1},
\]
which are conjugate and different from each other, because $b>0$, and $|\lambda_1| = |\lambda_2| =1$. Since $\beta'_4$ is real, the Jordan blocks associated to $\la_1$ and $\la_2$ in KCF$( \left( \beta_4' \right)^\top ,  \beta_4')$ are paired-up. Then, it suffices to determine the Jordan blocks associated to $\lambda_1$. For this, we prove that the geometric multiplicity of $\la_1$ is $1$, i.e., $\dim \, \mathcal{N} (\lambda_1 \left( \beta_4' \right)^\top +  \beta_4') =1$, via Lemma \ref{lemm.auxbetas}-part 2 with
$$A = \la_1 R^\top + R = \frac{2 b(b+\ii)}{1+b^2} \begin{bmatrix}
1 & -\ii \\ \ii & 1
\end{bmatrix}
\qquad \mbox{and} \qquad B = (\la_1 -1) S^\top
= \frac{2 (\ii b -1)}{1+b^2} \begin{bmatrix}
0 & -1 \\ 1 & 0
\end{bmatrix}.$$ The matrix $B$ is nonsingular,
\[
B^{-1} A = b \begin{bmatrix}
1 & -\ii \\ \ii & 1
\end{bmatrix} \qquad \mbox{and} \qquad \mathcal{C} (B^{-1} A) = \mbox{Span} \left\{ \begin{bmatrix}
1 \\ \ii
\end{bmatrix} \right\}.
\]
Therefore, $\mathcal{N} \,(A) \cap \mathcal{C} \,(B^{-1} A) = \{0\}$ and $\dim \, \mathcal{N} (\lambda_1 \left( \beta_4' \right)^\top +  \beta_4') = \dim \mathcal{N} (\la_1 R^\top + R) = 1$, by Lemma \ref{lemm.auxbetas}-part 2. As a consequence, the  KCF$( \left( \beta_4' \right)^\top ,  \beta_4')$ has only one Jordan block $J_k (-\la_1)$ associated with $\la_1$. This implies
\begin{align*}
( \left( \beta_4' \right)^\top ,  \beta_4') & \approx (I_k , J_k (-\la_1)) \oplus  (I_k , J_k (-\overline{\la_1}))  \approx \left(I_{2k} \, , \, C_{2k} \left(\frac{1-b^2}{1+b^2} \, , \, \frac{2 b}{1+b^2} \right)\right),
\end{align*}
where the last strict equivalence follows from Lemma \ref{cab_lemma}. Lemma \ref{lemm.strictrealpenc} guarantees that the expression above is real-KCF$( \left( \beta_4' \right)^\top ,  \beta_4')$. Finally, the result in the statement follows from the third case in the second row of the table in Theorem \ref{thm.1strealCFandKCF} (recall also Remark \ref{rem.recovereal}).
\end{proof}

\begin{lemma} \label{lemm.beta5}
The matrix $\beta_5'$ in Table {\rm\ref{correspondence_table}} is real-congruent to
$$
\displaystyle \widehat H_{4k}\left(\frac{1-(a^2+b^2)}{(1+|a|)^2+b^2}\, ,\, \frac{2|b|}{(1+|a|)^2+b^2}\right).
$$
\end{lemma}
\begin{proof} The proof proceeds again by first computing the real-KCF of $(\left( \beta_5'\right)^\top , \beta_5')$ and, then, by applying Theorem \ref{thm.1strealCFandKCF}. For simplicity, let us express $\beta_5'$ in \eqref{beta5} as
\begin{eqnarray} \label{eq.xybeta5}
\beta_5' = \begin{bmatrix}
0 & X \\
Y & 0
\end{bmatrix}, \qquad X, Y \in \R^{2k \times 2k}.
\end{eqnarray}
Observe that $X$ and $Y$ are symmetric and nonsingular, since $b\ne 0$. Thus,
\begin{align*}
( \left( \beta_5'\right)^\top , \beta_5') & = \left(  \begin{bmatrix}
0 & Y \\
X & 0
\end{bmatrix},  \begin{bmatrix}
0 & X \\
Y & 0
\end{bmatrix}\right)  \approx \left(  \begin{bmatrix}
Y & 0 \\
0 & X
\end{bmatrix},  \begin{bmatrix}
X & 0 \\
0 & Y
\end{bmatrix}\right) \\ & \approx \left(  \begin{bmatrix}
I_{2k} & 0 \\
0 & I_{2k}
\end{bmatrix},  \begin{bmatrix}
Y^{-1} X & 0 \\
0 &  (Y^{-1} X)^{-1}
\end{bmatrix}\right).
\end{align*}
Thus
\begin{equation} \label{eq.beta5oplus1}
\mbox{KCF} (\left( \beta_5'\right)^\top , \beta_5') = \mbox{KCF} (Y,X) \oplus \mbox{KCF} (X,Y),
\end{equation}
and $\mbox{KCF} (X,Y)$ can be obtained from $\mbox{KCF} (Y,X)$ changing each involved Jordan block $J_\ell (\mu)$ by $J_\ell (1/\mu)$. Therefore, we focus on computing $\mbox{KCF} (Y,X)$.
For this purpose, we follow an approach similar to that in Lemma \ref{lemm.beta4} for computing KCF$(\left( \beta_4'\right)^\top , \beta_4' )$. We consider the pencil
\begin{equation} \label{eq.pencilYX}
 \la Y + X =   \begin{bmatrix}
    &&&R'-\la T\\&&R'-\la T&(1-\la)S'\\&\iddots&\iddots\\R'-\la T&(1-\la)S'
\end{bmatrix}
\end{equation}
in the variable $\la$. This pencil has no minimal indices nor eigenvalues at infinity because $R'$ and $T$ are invertible. This implies that its KCF is determined by the Jordan blocks associated to its finite eigenvalues, which are the roots of the polynomial
\begin{align*}
\det (R'-\la T) & = - (\, b^2 (\la-1)^2 + (\la (1-a) + (a+1))^2 \, ) \\ & = - (\, \la (1-a) + (a+1) + \ii b(\la -1)) \, \, (\la (1-a) + (a+1) - \ii b(\la -1)  ),
\end{align*}
each with algebraic multiplicity $k$. These roots are
$$
\la_1 = \frac{a+1 - \ii b}{a-1 - \ii b} = \frac{a^2 + b^2 -1 + \ii \, 2 b}{(a-1)^2 + b^2} \qquad \mbox{and} \qquad \la_2 = \overline{\la_1}
$$
and are complex conjugate to each other, because $b\ne 0$, and $|\la_1| = |\la_2| \ne 1$, because, $a\ne 0$. Since $Y$ and $X$ are real, the Jordan blocks associated to $\la_1$ and $\la_2$ in KCF$( Y ,  X)$ are paired-up. So, it suffices to determine the Jordan blocks associated to $\lambda_1$. For this, we prove that the geometric multiplicity of $\la_1$ is $1$, i.e., $\dim \, \mathcal{N} (\lambda_1 Y +  X) =1$, via Lemma \ref{lemm.auxbetas}-part 2 with
$$A = R' -\la_1 T = -\frac{2 b}{a-1-\ii b} \begin{bmatrix}
1 & \ii \\ \ii & -1
\end{bmatrix}
\qquad \mbox{and} \qquad B = (1-\la_1) S'.$$ The matrix $B$ is nonsingular,
\[
B^{-1} A = -b \begin{bmatrix}
-\ii & 1 \\ -1 & -\ii
\end{bmatrix} \qquad \mbox{and} \qquad \mathcal{C} (B^{-1} A) = \mbox{Span} \left\{ \begin{bmatrix}
1 \\ -\ii
\end{bmatrix} \right\}.
\]
Therefore, $\mathcal{N} \,(A) \cap \mathcal{C} \,(B^{-1} A) = \{0\}$ and $\dim \, \mathcal{N} (\lambda_1 Y +  X) = \dim \mathcal{N} (R' - \la_1 T) = 1$, by Lemma \ref{lemm.auxbetas}-part 2. As a consequence, the  KCF$(Y,  X)$ has only one Jordan block $J_k (-\la_1)$ associated with $\la_1$. This implies
KCF$(Y,X) = (I_k , J_k (-\la_1)) \oplus  (I_k , J_k (-\overline{\la_1}))$. Combining this with \eqref{eq.beta5oplus1} and Lemma \ref{cab_lemma}, we get
\begin{align*}
( \left( \beta_5' \right)^\top ,  \beta_5') & \approx (I_k , J_k (-\la_1)) \oplus  (I_k , J_k (-\overline{\la_1})) \oplus  (I_k , J_k \left(- 1/\la_1 \right) ) \oplus  \left(I_k , J_k \left(-1/\overline{\la_1} \right)\right)
\\  & \approx (I_{2k} \, , \, C_{2k} \left(\frac{1-(a^2 + b^2)}{(a-1)^2+b^2} \, , \, \frac{2 \, |b|}{(a-1)^2+b^2} \right)) \\ &  \phantom{\approx}
\oplus
 (I_{2k} \, , \, C_{2k} \left(\frac{1-(a^2 + b^2)}{(a+1)^2+b^2} \, , \, \frac{2 \, |b|}{(a+1)^2+b^2} \right)).
\end{align*}
Lemma \ref{lemm.strictrealpenc} guarantees that the expression above is the real-KCF$( \left( \beta_5' \right)^\top ,  \beta_5')$. The result follows from the last row in the table of Theorem \ref{thm.1strealCFandKCF}, taking into account that from the two  $C_{2k} (c,d)$ above, we choose the parameters $c$ and $d$ such that $c^2 + d^2 <1$, i.e., the ones with largest denominator.
\end{proof}

\vspace*{1cm}

\appendix

\noindent
{\LARGE \bf Appendices}

\section{Proof of Lemma \ref{equivalencegammas.th}} \label{append.proof-lemma}

\begin{proof}
We will first see that the nonzero entries of $(P_kS_k)^\top \Gamma_k(P_kS_k)$ and $\widetilde\Gamma_k$ are placed in the same positions. For this, first notice that the nonzero entries of $\Gamma_k$ are just $1$ or $-1$, and they are placed in the following positions:
\begin{itemize}
    \item  $(k-i+1,i)$, for $i=1,\hdots,k$ (placed in the main anti-diagonal).
    \item  $(k-i+2,i)$, for $i=2,\hdots,k$ (placed below the main anti-diagonal).
\end{itemize}
These entries $1$ or $-1$ are also the nonzero entries of $(P_kS_k)^\top \Gamma_k(P_kS_k)$, since $P_k$ is a permutation matrix and $S_k$ is just a change of signs matrix, but they are placed in different positions. Now, we are going to identify the positions of these nonzero entries in $(P_kS_k)^\top \Gamma_k(P_kS_k)$. We analyze separately the cases $k$ even and $k$ odd.

\medskip

$\blacktriangleright$ $k$ even: Let us first identify the positions of the nonzero entries in the matrix $P_k^\top\Gamma_k P_k$.

The permutation corresponding to $P_k$, described in the statement, acts as follows on the $i$th column-index:
$$
\begin{array}{ccc}
     i&\longmapsto&\left\{\begin{array}{cc}
          k-2i+2&\mbox{for $i=1,\hdots,\frac{k}{2}$},  \\
          2i-k-1&\mbox{for $i=\frac{k}{2}+1,\hdots,k$}.
     \end{array}\right.
\end{array}
$$
As a consquence, when applying the permutation to the indices corresponding to the nonzero entries of $\Gamma_k$, these entries go to the following positions:
\begin{eqnarray}
     (k-i+1,i)&\longmapsto&\left\{\begin{array}{lc}
          (A):\ (k-2i+1,k-2i+2)&\mbox{for $i=1,\hdots,\frac{k}{2}$},  \\
          (B):\ (2i-k,2i-k-1)&\mbox{for $i=\frac{k}{2}+1,\hdots,k$.}
     \end{array}\right. \label{entrieskeven1} \\
     (k-i+2,i)&\longmapsto&\left\{\begin{array}{lc}
         (C):\  (k-2i+3,k-2i+2)&\mbox{for $i=2,\hdots,\frac{k}{2}$},  \\
          (E):\ (1,1)&\mbox{for $i=\frac{k}{2}+1$},\\
          (D):\ (2i-k-2,2i-k-1)&\mbox{for $i=\frac{k}{2}+2,\hdots,k$}.
     \end{array}\right.\label{entrieskeven2}
\end{eqnarray}
The indices $(A)$ and $(D)$ above altogether run through all the entries in the upper diagonal, namely those with indices $(i,i+1)$, for $i=1,\hdots,k-1$. Similarly, the indices $(B)$ and $(C)$ correspond to all the entries in the lower diagonal, namely those with indices $(i+1,i)$, for $i=1,\hdots,k-1$. These entries, together with $(E)=(1,1)$ are all the nonzero entries of the matrix $P_k^\top\Gamma_k P_k$. Note that these are, precisely, the same positions which contain all the nonzero entries of $\widetilde\Gamma_k$.

Now, let us identify the positions with entries $-1$ in $(P_kS_k)^\top \Gamma_k(P_kS_k)$. Note first that the entries equal to $-1$ in $\Gamma_k$ are placed in the positions $(k-2i+1,2i)$, for $i=1,\hdots,k/2$ (those in the main anti-diagonal), and $(k-2i+1,2i+1)$, for $i=1,\hdots,(k/2)-1$ (those below the main anti-diagonal). Now, we will keep track of these entries when applying the permutation corresponding to $P_k$. For this, we distinguish the cases $k\equiv0$ (mod $4$) and $k\equiv2$ (mod $4$).

\begin{itemize}
    \item $k\equiv0$ (mod $4$): Replacing $i$ by $2i$ in \eqref{entrieskeven1} and by $2i+1$ in \eqref{entrieskeven2}, the permutation corresponding to $P_k$ acts as follows on the previous indices:
    $$
    \begin{array}{ccc}
        (k-2i+1,2i) &\longmapsto&\left\{\begin{array}{lc}
            (A):\ (k-4i+1,k-4i+2) &\mbox{for $i=1,\hdots,\frac{k}{4}$},  \\
             (B):\ (4i-k,4i-k-1)&\mbox{for $i=\frac{k}{4}+1,\hdots,\frac{k}{2}$}.
        \end{array}\right.  \\
         (k-2i+1,2i+1)&\longmapsto&\left\{\begin{array}{lc}
              (C):\ (k-4i+1,k-4i)&\mbox{for $i=1,\hdots,\frac{k}{4}-1$},  \\
              (E):\ (1,1)&\mbox{for $i=\frac{k}{4}$},\\
              (D):\ (4i-k,4i-k+1)&\mbox{for $i=\frac{k}{4}+1,\hdots,\frac{k}{2}-1$.}
         \end{array}\right.
    \end{array}
    $$
    The first indices in $(A),(C),$ and $(E)$ above altogether are $1,5,\hdots,k-3$, namely, all indices $i\equiv1$ (mod $4$) between $1$ and $k$. Therefore, the corresponding entries are all the nonzero entries of $P_k^\top\Gamma_kP_k$ in these rows (two entries per row).

    Similarly, the first indices in $(B)$ and $(D)$ above altogether are $4,8,\hdots,k$ (the last one appearing only once in $(B)$ for $i=k/2$). Therefore, the positions $(B)$ and $(C)$ correspond to all the nonzero entries in $P_k^\top\Gamma_kP_k$ in rows with indices $i\equiv0$ (mod $4$).

    Summarizing, all the entries equal to $-1$ in the matrix $P_k^\top\Gamma_kP_k$ are in the rows with indices $i\equiv0,1$(mod $4$), and all the nonzero entries in these rows are equal to $-1$.

    The change of signs matrix $S_k$ changes the sign of all rows and columns with indices $i\equiv3$ (mod $4$). The change of sign in the rows turns into $-1$ all entries $(i,i-1)$ and $(i,i+1)$ with $i\equiv3$ (mod $4$), and the change of sign in the columns changes the signs of the entries $(i+1,i)$ and $(i-1,i)$ with $i\equiv3$ (mod $4$), namely $i+1\equiv0$ (mod $4$) and $i-1\equiv2$ (mod $4$). Therefore, in $(P_kS_k)^\top \Gamma_k(P_kS_k)=S_k^\top (P_k^\top \Gamma_kP_k)S_k$ all entries in the upper diagonal, namely those with indices $(i,i+1)$, are equal to $-1$, whereas only the entries in the positions $(i,i-1)$ with $i\equiv1,3$ (mod $4$) are equal to $-1$ (the ones in the positions $(i,i-1)$ with $i\equiv0,2$ (mod $4$) are all equal to $1$ instead). In other words, $(P_kS_k)^\top \Gamma_k(P_kS_k)=-\widetilde\Gamma_k$, as claimed.

    \item $k\equiv2$ (mod $4$): Replacing, again, $i$ by $2i$ in \eqref{entrieskeven1} and by $2i+1$ in \eqref{entrieskeven2}, the permutation $P_k$ acts on the indices corresponding to the negative entries of $\Gamma_k$ in a similar way, though the range of the indices now is slightly different, namely:
     $$
    \begin{array}{ccc}
        (k-2i+1,2i) &\longmapsto&\left\{\begin{array}{lc}
            (A):\ (k-4i+1,k-4i+2) &\mbox{for $i=1,\hdots,\frac{k-2}{4}$},  \\
             (B):\ (4i-k,4i-k-1)&\mbox{for $i=\frac{k+2}{4},\hdots,\frac{k}{2}$}.
        \end{array}\right.  \\
         (k-2i+1,2i+1)&\longmapsto&\left\{\begin{array}{lc}
              (C):\ (k-4i+1,k-4i)&\mbox{for $i=1,\hdots,\frac{k-2}{4}$},  \\
              (D):\ (4i-k,4i-k+1)&\mbox{for $i=\frac{k+2}{4},\hdots,\frac{k}{2}-1$.}
         \end{array}\right.
    \end{array}
    $$

    Now, the indices in $(A)$ and $(C)$ correspond to the two nonzero entries in the rows $3,7,\hdots,k-3$, namely all rows with indices $i\equiv3$ (mod $4$). Similarly, the entries in $(B)$ and $(D)$ correspond to the nonzero entries in the rows $2,6,\hdots,k$ (which are two entries per row, except in row $k$ where there is only one) , namely those rows with indices $i\equiv2$ (mod $4$). Summarizing, the $-1$ entries in $P_k^\top\Gamma_kP_k$ are those in rows with indices $i\equiv2,3$ (mod $4$).

    When we multiply on the left and on the right the matrix $P_k^\top\Gamma_kP_k$ by $S_k$, we introduce a change of sign in the rows and columns with indices $i\equiv3$ (mod $4$). The change of sign in the rows turn the $-1$ into $1$ in the entries placed in all rows with indices $i\equiv3$ (mod $4$). Now, the change of sign in the columns changes the sign of the entries with indices $(i-1,i)$ and $(i+1,i)$, with $i\equiv3$ (mod $4$), namely $i-1\equiv2$ (mod $4$) and $i+1\equiv0$ (mod $4$). Therefore the entries of $(P_kS_k)^\top \Gamma_k(P_kS_k)=S_k^\top (P_k^\top \Gamma_kP_k)S_k$ in the positions $(i,i-1)$, for $i\equiv0,2$ (mod $4)$, are $-1$, and these are the only entries which are equal to $-1$ in this matrix. In other words, $(P_kS_k)^\top \Gamma_k(P_kS_k)=\widetilde\Gamma_k$, as claimed.
\end{itemize}

\medskip

$\blacktriangleright$ $k$ odd: In this case, the permutation corresponding to $P_k$ acts as follows on the $i$th column-index:
$$
\begin{array}{ccc}
     i&\longmapsto&\left\{\begin{array}{cc}
          k-2i+2&\mbox{for $i=1,\hdots,\frac{k+1}{2}$},  \\
          2i-k-1&\mbox{for $i=\frac{k+3}{2},\hdots,k$}.
     \end{array}\right.
\end{array}
$$
As a consequence, the permutation applied to the matrix $\Gamma_k$ acts as follows on the positions corresponding to the nonzero entries of $\Gamma_k$:
\begin{eqnarray}
     (k-i+1,i)&\longmapsto&\left\{\begin{array}{lc}
          (A):\ (k-2i+1,k-2i+2)&\mbox{for $i=1,\hdots,\frac{k-1}{2}$},  \\
           (E):\ (1,1)&\mbox{for $i=\frac{k+1}{2}$},\\
          (B):\ (2i-k,2i-k-1)&\mbox{for $i=\frac{k+3}{2},\hdots,k$.}
     \end{array}\right.  \label{entrieskodd1}\\
     (k-i+2,i)&\longmapsto&\left\{\begin{array}{lc}
         (C):\  (k-2i+3,k-2i+2)&\mbox{for $i=2,\hdots,\frac{k+1}{2}$},  \\
          (D):\ (2i-k-2,2i-k-1)&\mbox{for $i=\frac{k+3}{2},\hdots,k$}.
     \end{array}\right.\label{entrieskodd2}
\end{eqnarray}
The indices $(A)$ and $(D)$ altogether correspond to the entries in the upper diagonal of $P_k^\top\Gamma_kP_k$, namely those with indices $(i,i+1)$, for $i=1,\hdots,k-1$. Similarly, the indices $(B)$ and $(C)$ together correspond to $(i,i-1)$, for $i=2,\hdots,k$, i.e., to the entries in the lower diagonal of $P_k^\top \Gamma_k P_k$. These indices, together with $(E)=(1,1)$ correspond to the nonzero entries of $P_k^\top\Gamma_kP_k$, as well as those of $\widetilde\Gamma_k$.

As in the case $k$ even, we are now going to identify the positions of the $-1$ entries in $(P_kS_k)^\top \Gamma_k(P_kS_k)$. Starting again from the entries equal to $-1$ in $\Gamma_k$, which are placed in the positions $(k-2i+1,2i)$ and $(k-2i+1,2i+1)$, for $i=1,\hdots,(k-1)/2$, we keep track of these entries after applying the permutation corresponding to the matrix $P_k$ and then applying the change of signs corresponding to $S_k$. We analyze separately the cases $k\equiv1$ (mod $4$) and $k\equiv3$ (mod $4$).

\begin{itemize}
    \item $k\equiv1$ (mod $4$): Replacing $i$ by $2i$ in \eqref{entrieskodd1} and by $2i+1$ in \eqref{entrieskodd2}, the permutation corresponding to $P_k$ acts as follows on the previous indices:
$$
    \begin{array}{ccl}
        (k-2i+1,2i) &\longmapsto&\left\{\begin{array}{lc}
            (A):\ (k-4i+1,k-4i+2) &\mbox{for $i=1,\hdots,\frac{k-1}{4}$},  \\
             (B):\ (4i-k,4i-k-1)&\mbox{for $i=\frac{k+3}{4},\hdots,\frac{k-1}{2}$}.
        \end{array}\right.  \\
         (k-2i+1,2i+1)&\longmapsto&\left\{\begin{array}{lc}
              (C):\ (k-4i+1,k-4i)&\mbox{for $i=1,\hdots,\frac{k-1}{4}$},  \\
              (D):\ (4i-k,4i-k+1)&\mbox{for $i=\frac{k+3}{4},\hdots,\frac{k-1}{2}$.}
         \end{array}\right.
    \end{array}
    $$
    These are all nonzero entries of the matrix $P_k^\top \Gamma_kP_k$ in the following rows: the entries $(A)$ and $(C)$ together are in the rows $2,6,\hdots,k-3$, namely all rows with indices $i\equiv2$ (mod $4$), whereas the entries $(B)$ and $(D)$ together are the nonzero entries in the rows $3,7,\hdots,k-2$, namely the rows with indices $i\equiv3$ (mod $4$). Therefore, the entries $(A)-(D)$ correspond to the nonzero entries of $P_k^\top\Gamma_kP_k$ in the rows with indices $i\equiv2,3$ (mod $4$). Therefore, we are in the same situation as for $k\equiv2$ (mod $4$), so $(P_kS_k)^\top\Gamma_k(P_kS_k)=\widetilde\Gamma_k$.

    \item $k\equiv3$ (mod $4$):  Replacing $i$ by $2i$ in \eqref{entrieskodd1} and by $2i+1$ in \eqref{entrieskodd2}, this time we have:
$$
    \begin{array}{ccl}
        (k-2i+1,2i) &\longmapsto&\left\{\begin{array}{lc}
            (A):\ (k-4i+1,k-4i+2) &\mbox{for $i=1,\hdots,\frac{k-3}{4}$},  \\
             (E):\ (1,1)&\mbox{for $i=\frac{k+1}{4}$,}\\
             (B):\ (4i-k,4i-k-1)&\mbox{for $i=\frac{k+1}{4}+1,\hdots,\frac{k-1}{2}$}.
        \end{array}\right.  \\
         (k-2i+1,2i+1)&\longmapsto&\left\{\begin{array}{lc}
              (C):\ (k-4i+1,k-4i)&\mbox{for $i=1,\hdots,\frac{k-3}{4}$},  \\
              (D):\ (4i-k,4i-k+1)&\mbox{for $i=\frac{k+1}{4},\hdots,\frac{k-1}{2}$.}
         \end{array}\right.
    \end{array}
    $$
    The entries $(A)-(E)$ in this case correspond to all the nonzero entries of $P_k^\top\Gamma_kP_k$ in the rows with indices $i\equiv0,1$ (mod $4$). More precisely, the first indices in $(A)$ and $(C)$ together are equal to $4,8,\hdots,k-3$, and those of $(B), (D)$, and $(E)$ together are equal to $1,5,\hdots,k-2$ (in particular, the index $1$ comes from $(E)$, and from $(D)$ for $i=(k+1)/4$). Therefore, we are in the same situation as in the case $k\equiv0$ (mod $4$), so, again, $(P_kS_k)^\top\Gamma_k(P_kS_k)=-\widetilde\Gamma_k$.
\end{itemize}
\end{proof}

\section{Canonical forms of degenerate blocks $\beta_4'$ and $\beta_5'$ in Theorem \ref{thm.lw96}} \label{append.degenerate}
For completeness, the next table presents the canonical forms in Theorem \ref{main_th} of the matrices $\beta_4'$ and $\beta_5'$ in \eqref{beta4} and \eqref{beta5} if $a =0$ or $b =0$. We omit the proof for brevity.

\medskip

\begin{center}
{\renewcommand{\arraystretch}{2.4}
    \renewcommand{\tabcolsep}{0.15cm}
    \begin{tabular}{|c|c|}\hline
         \begin{tabular}{l}
         Degenerate variant of \\[-0.5cm]
         block $B$ in Theorem \ref{thm.lw96}
         \end{tabular}
         &is real-congruent to direct sum of blocks in Theorem \ref{main_th}  \\\hline\hline
         $\beta_4'$&\begin{tabular}{ll} $\displaystyle \Gamma_k\otimes C (0,1)$ & if $b=0$ and $k$ even
         \\[0.1cm]
         $\displaystyle \left(\Gamma_k\otimes \left(\varepsilon \, (-1)^\ell \right) \right) \oplus
         \left(\Gamma_k\otimes \left(\varepsilon \, (-1)^\ell \right) \right)
         $
         &if $b=0$ and $k = 2 \ell +1$\\[0.3cm] \end{tabular}\\[0.2cm] \hline
         $\beta_5'$& \begin{tabular}{l}
               $\left(\Gamma_k \otimes C(0,1) \right) \oplus  \left(\Gamma_k \otimes C(0,1) \right)$  \; if $b=0$, $a=0$, and $k$ even \\
               $\Gamma_k \oplus \Gamma_k \oplus (-\Gamma_k) \oplus (-\Gamma_k)$  \;  if $b=0$, $a=0$, and $k$ odd \\
               $J_{2k} (0) \oplus J_{2k} (0)$  \; if $b=0$, $a=\pm 1$\\
              $\displaystyle H_{2k}\left(\frac{1-|a|}{1+|a|} \right) \oplus H_{2k}\left(\frac{1-|a|}{1+|a|} \right)$  \; if $b=0$, $a\ne 0$, $a\ne \pm1$ \\[0.2cm]
              $(\Gamma_k\otimes C \left(\pm \frac{|b|}{\sqrt{1+b^2}} \, ,\, \frac{1}{\sqrt{1+b^2}} \right)) \oplus (\Gamma_k\otimes C \left(\pm \frac{|b|}{\sqrt{1+b^2}} \, ,\, \frac{1}{\sqrt{1+b^2}} \right))$  \\ \qquad \qquad if $b\ne 0$, $a=0$, $k$ even
              \\[0.3cm]
              $(\Gamma_k\otimes C \left(\pm \frac{1}{\sqrt{1+b^2}}  \, ,\, \frac{|b|}{\sqrt{1+b^2}}  \right) ) \oplus (
              \Gamma_k\otimes C \left(\pm \frac{1}{\sqrt{1+b^2}}  \, ,\, \frac{|b|}{\sqrt{1+b^2}}  \right))$ \\
              \qquad \qquad if $b\ne 0$, $a=0$, $k$ odd \\[0.3cm]
              \end{tabular}
         \\[0.3cm] \hline
    \end{tabular}}
\end{center}

\bigskip

\noindent{\bf Acknowledgments}. This work has been partially supported by grants PID2023-147366NB-I00 funded by MICIU/AEI/10.13039/501100011033 and FEDER/UE, and RED2022-134176-T.

\bibliographystyle{plain}

\end{document}